\documentclass{amsart}[10pt]

\usepackage[a4paper,margin=1.19in]{geometry}  
\usepackage{graphicx}              
\usepackage{hyperref}
\usepackage{amsmath}               
\usepackage{amsfonts}              
\usepackage{amsthm}                
\usepackage{amssymb}
\usepackage[utf8]{inputenc}
\usepackage{lmodern}
\usepackage{enumerate}
\usepackage{marginnote}
\usepackage{tikz}
\usepackage{tikz-cd}
\usetikzlibrary{decorations.markings}
\tikzset{degil/.style={
            decoration={markings,
            mark= at position 0.5 with {
                  \node[transform shape] (tempnode) {$\backslash$};
                  }
              },
              postaction={decorate}
}
}
\usetikzlibrary{decorations.text, calc, matrix, arrows, decorations.pathreplacing}
\usepackage[all]{xy}
\usepackage{forloop}
\usepackage{soul}
\usepackage{color}
\usepackage{leftidx}
\usepackage[capitalise]{cleveref}

\newtheorem{thm}{Theorem}[section]
\newtheorem{lemma}[thm]{Lemma}
\newtheorem{prop}[thm]{Proposition}
\newtheorem{cor}[thm]{Corollary}

\theoremstyle{definition} 
\newtheorem{defin}[thm]{Definition}
\newtheorem{example}[thm]{Example}
\newtheorem{remark}[thm]{Remark}

\newcommand{\bigslant}[2]{{\raisebox{.2em}{$#1$}\left/\raisebox{-.2em}{$#2$}\right.}}
\newcommand\mydef{\mathrel{\overset{\makebox[0pt]{\mbox{\normalfont\tiny def}}}{=}}}
\newcommand{\comp}{\kern0.5ex\vcenter{\hbox{$\scriptstyle\circ$}}\kern0.5ex}

\newcommand{\bb}[1]{\mathbb{#1}}
\newcommand{\cc}[1]{\mathcal{#1}}
\newcommand{\cO}{\mathcal O}
\newcommand{\cN}{\mathcal N}
\newcommand{\cT}{\mathcal T}
\newcommand{\cE}{\mathcal E}
\newcommand{\cF}{\mathcal F}
\newcommand{\cG}{\mathcal G}
\newcommand{\cH}{\mathcal H}

\newcommand{\cI}{\mathcal I}
\newcommand{\cK}{\mathcal K}

\newcommand{\cJ}{\mathcal J}
\newcommand{\Affine}{{\mathbb A}}
\newcommand{\isom}{\simeq}
\newcommand{\eps}{\varepsilon}
\newcommand{\ideal}[1]{{\mathfrak{#1}}}
\newcommand{\F}{F^{\#}}
\newcommand{\ceil}[1]{{\lceil #1 \rceil}}
\newcommand{\floor}[1]{{\lfloor #1 \rfloor}}
\DeclareMathOperator{\id}{id}
\DeclareMathOperator{\PP}{\bb{P}}
\DeclareMathOperator{\ZZ}{\bb{Z}}
\DeclareMathOperator{\NN}{\bb{N}}
\DeclareMathOperator{\Proj}{Proj}

\DeclareMathOperator{\Ob}{Ob}
\DeclareMathOperator{\Tan}{T}
\DeclareMathOperator{\Ker}{Ker}
\DeclareMathOperator{\Coker}{Coker}

\DeclareMathOperator{\sex}{SmallExt}
\DeclareMathOperator{\QCoh}{QCoh}
\DeclareMathOperator{\Coh}{Coh}
\DeclareMathOperator{\Tor}{Tor}

\newcommand{\Frob}[1]{{#1}^{(1)}}

\newcommand{\Cone}{\mathsf{Cone}}

\newcommand{\Ext}{\mathrm{Ext}}

\newcommand{\RcHom}{\mathbf{R}{\mathcal H}om}

\newcommand{\Spec}{\mathrm{Spec}}
\newcommand{\sSpec}{{\cc{S} pec}}

\newcommand{\Cot}{\mathbb{L}}
\newcommand{\Hilb}{\mathsf{Hilb}}
\newcommand{\Def}{\mathsf{Def}}
\newcommand{\Art}{\mathbf{Art}}
\newcommand{\Sets}{\mathbf{Sets}}

\newcommand{\ra}{\longrightarrow}
\newcommand{\mplus}{+_{\varphi}}
\newcommand{\mmul}{\cdot_{\varphi}}
\newcommand{\isomto}{\xrightarrow{
   \,\smash{\raisebox{-0.65ex}{\ensuremath{\scriptstyle\sim}}}\,}}

\newcommand{\wt}[1]{\widetilde{#1}}  

\begin{document}
\author{Maciej Zdanowicz}
\address{Instytut Matematyki UW, Banacha 2, 02-097 Warszawa, Poland}
\email{mez@mimuw.edu.pl}
\thanks{The author was supported by Polish National Science Centre (NCN) contract number 2014/13/N/ST1/02673.}
\title{Liftability of singularities and their Frobenius morphism modulo $p^2$}

\begin{abstract}
We investigate the $W_2(k)$-liftability of singular schemes.  We prove constructibility of the locus of $W_2(k)$-liftable schemes in a flat family $X \to S$. Moreover, we construct an explicit $W_2(k)$-lifting of a Frobenius split scheme $X$ over a perfect field $k$, reproving Bhatt's existential result.  Furthermore, we study existence of liftings of the Frobenius morphism. In particular, we prove that in dimension $n \geq 4$ ordinary double points do not admit a $W_2(k)$-lifting compatible with Frobenius, and that canonical surface singularities are Frobenius liftable.  Combined with Bhatt's results, the latter result implies that the crystalline cohomology groups over $k$ of surfaces with canonical singularities are not finite dimensional.  As a corollary of our results, we provide a thorough comparison between the notions of $W_2(k)$-liftability, Frobenius liftability and classical $F$-singularity types.
\end{abstract}

\maketitle


\section{Introduction}

In the seminal paper \cite{deligne-illusie}, Deligne and Illusie observed that a smooth variety $X$ defined over a perfect field $k$ of characteristic $p$ possesses many desired properties, which are standard in characteristic $0$, provided it admits a flat lifting to $W_2(k)$, i.e., a unique $\ZZ/p^2$-flat extension of $k$.  This observation turned out to be very fruitful and leads to many further results overcoming pathological behaviour of positive characteristic geometry of \emph{smooth} varieties (see, e.g., \cite{ogus-vologodsky,langer}).  The crucial part of the arguments given in the papers mentioned was the existence of a local lifting of the Frobenius morphism mod $p^2$.

In this paper, we investigate the existence of $W_2(k)$ and Frobenius liftings in the more general setting of singular schemes.  In this situation, contrary to the case of smooth schemes, a $W_2(k)$-lifting need not to exist even affine locally.  The simplest instance of that phenomenon, in fact $0$-dimensional, was given without a proof in \cite[Remark 3.16]{bhatt_derived_derham}.  We present a detailed analysis of a slightly more general example in \cref{sec:quadric_example}.  

Furthermore, in \cite{bhatt_torsion} the author generalizes the results of \cite{deligne-illusie} under the assumption of global Frobenius liftability (see \cref{def:lifting}), and consequently proves that the crystalline cohomology over $k$ of Frobenius liftable locally complete intersection schemes is not finitely generated. 

The aforementioned examples and results suggest the following two basic question:

\begin{description}
\item[Question 1] What are the necessary conditions for a scheme to possess a $W_2(k)$-lifting?
\item[Question 2] For a given scheme, does there exist a $W_2(k)$-lifting together with a compatible lifting of the Frobenius morphism?
\end{description}

General answers to the questions were given in \cite{cotangent,mehta_srinivas,joshi} and stated that above problems can be expressed in purely cohomological manner.  Moreover, in \cite{liedtke_satriano} the authors investigated the first question in the setting of birational geometry.  Our contribution to answering the above questions is the following.

As far as the first question is concerned, in \cref{sec:families} we analyse how the property of $W_2(k)$-liftability behaves in families $f \colon X \to S$.  More precisely, we show, that under mild assumptions suitable for deformation-theoretic considerations (satisfied if e.g. $f$ is proper or affine), the locus of Witt vector liftable fibres $X_s \mydef f^{-1}(s)$ is constructible (see \cref{thm:constructible}).  Furthermore, we reprove the result that Frobenius split schemes are $W_2(k)$-liftable.  The advantage of our approach is that we provide the lifting explicitly and functorially with respect to a splitting $\varphi \colon F_{X*} \cO_X \to \cO_X$, in fact as a certain subscheme of the Witt scheme $(X,W_2(\cO_X))$.  The details of our construction are presented in \cref{sec:frobenius_split_witt}.   

Moreover, we give a new construction of a non-liftable affine scheme.  In fact, the construction is given by a cone over a suitable embedding on a non-liftable projective scheme (for the proof see \cref{lem:comparing_hilb_def} in the \cref{appendix:deformation}).  Examples of non-liftable projective schemes are given, e.g., by Raynaud \cite{raynaud} and Mukai \cite{mukai}.

Concerning the second question, we derive a computationally feasible criterion for Frobenius liftability of complete intersection affine schemes (see \cref{lem:complete_intersections_criterion}).  Consequently, we apply the criterion in the case of ordinary double points (see \cref{thm:cones_bhatt}), and thus answer the question asked by Bhatt in \cite[Remark 3.14]{bhatt_torsion}.  As another application of our criterion, in \cref{sec:canonical_singularities} we prove that canonical surface singularities are Frobenius liftable.  Finally, in \cref{sec:quotient} we present a general approach for proving Frobenius liftability of tame reductions of quotient singularities.  The method is based on the \emph{spreading out} technique combined with functoriality of obstruction classes for lifting morphisms.

The above results allow us to address the following question: 
\begin{description}
\item[Question 3] What is the relation between $W_2(k)$-liftability and Frobenius liftability, and classical notions defined in terms of Frobenius action?
\end{description}

Our contribution can be summarized by the following diagram which describes all potential relations between the classical $F$-singularities and the notion considered in this paper:

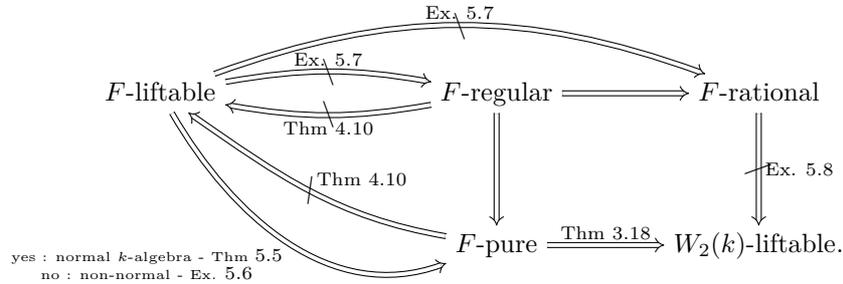
\begin{center}
\begin{tikzcd}[column sep=large,arrows=Rightarrow]
\text{$F$-liftable} \arrow[degil,out=20,in=160]{rrr}{\text{Ex. } \ref{ex:f-liftable_not_f-regular}} \arrow[out=10,in=170,degil]{rr}{\text{Ex. } \ref{ex:f-liftable_not_f-regular}} \arrow[out=-60,in=-160]{rrdd}[swap]{\substack{\text{yes : normal $k$-algebra - Thm } \ref{lem:flift_normal} \\ \text{no : non-normal - Ex. } \ref{ex:lift_not_imply_split}}} & & \text{$F$-regular} \arrow{dd}{} \arrow[out=190,in=-10,degil]{ll}{\text{Thm } \ref{thm:cones_bhatt}} \arrow{r}{} & \text{$F$-rational} \arrow[degil]{dd}{\text{Ex. } \ref{ex:frational_not_imply_w2k}
} \\
& & & \\ 
& & \text{$F$-pure} \arrow[out=170,in=-35,degil]{lluu}[swap]{\text{Thm } \ref{thm:cones_bhatt}} \arrow{r}{\text{Thm } \ref{thm:split_w2k}} &\text{$W_2(k)$-liftable}. \\
\end{tikzcd}
\end{center}

The counterexamples corresponding to most of the strikethrough implications in the diagram are provided in \cref{sec:relation_f-singularities}

As a byproduct of our considerations, we formalise the foregoing functoriality properties of obstruction classes for deformation of schemes and their morphism.  In particular, we analyse their behaviour in the case of affine schemes.  The results are given in \cref{appendix:deformation}.  Apart from explicit statements and proofs concerning functoriality, the whole theory can be found in the standard reference \cite{cotangent}.

In the course of our consideration we faced the following problem:
\begin{description}
\item[Open question] Does there exist a finite \'{e}tale covering $\pi : Y \to X$ of schemes locally of finite type over $k$ such that $Y$ is $W_2(k)$-liftable and $X$ is not.
\end{description}
Note that by \cref{lem:func_w2k} the degree of $\pi$ would necessarily be divisible by $p$ and by \cref{lem:affine_deformation_etale_local} the schemes need not be affine.

\subsection{Notation and basic definitions}\label{sec:notation}

Here we present a few general assumptions on schemes and their morphisms.  Moreover, we give basic definitions concerning Frobenius morphism and Witt vector liftability.
\subsubsection{Generalities on schemes and morphisms}
Throughout the following work $k$ is a perfect field of characteristic $p>0$.  All rings we consider are commutative, noetherian and with identity.  Unless otherwise stated the schemes and morphisms between them are quasi-compact, quasi-separated and essentially of finite type over $k$ (in particular, the schemes are noetherian).  For any morphism $f : X \to Y$, the associated mapping of sheaves of rings is denoted by $f^\# : \cc{O}_Y \to f_*\cc{O}_X$.  For any $k$-scheme $X$ by $F_X : X \to X$ we mean the absolute Frobenius morphism, i.e., the identity on the level of topological spaces and $F^{\#}_X :\cc{O}_X \to F_{X*}\cc{O}_X$ defined by $F^{\#}_X(f) = f^p$.  Moreover, for any morphism $f : X \to S$ by 
\[
F_{X/S} : X \to X^{(S)} \mydef X \times_{F_S} S
\] 
we denote the relative Frobenius morphism defined by the following diagram with a cartesian right square:
\begin{displaymath}
\xymatrix{
	X \ar@/_-2.0pc/[rrrr]^{F_X} \ar[rr]^-(0.65){F_{X/S}}\ar[rrd]^{f} & & X^{(S)} \ar[d]^{f^{(S)}}\ar[rr]^-(0.35){W_{X/S}} & & X\ar[d]^{f} \\
	& & S \ar[rr]^{F_S} & & S.}
\end{displaymath}  

By $W_n(k)$ we denote the ring of Witt vectors of length $n$ and by $\sigma : W_n(k) \to W_n(k)$ the associated Frobenius morphism given by the formula $(a_0,\ldots,a_{n-1}) \mapsto (a_0^p,\ldots,a_{n-1}^p)$

\subsubsection{Frobenius splitting}

By a Frobenius splitting of $X$ we mean a homomorphism $\varphi : F_{X*}\cc{O}_X \to \cc{O}_X$ of $\cc{O}_X$-modules which splits the canonical map $\F_{X} : \cO_X \to F_{X*}\cO_X$.  We say that a scheme $X$ is \emph{Frobenius-split} if there exists an associated Frobenius splitting.  For an exhaustive treatment of Frobenius splittings the reader is referred to \cite{brion_kumar}.

\subsubsection{Lifting of schemes and Frobenius morphism}

Let $X/S$ be a morphism of schemes and $S \to \wt{S}$ a closed immersion given by a square-zero ideal.

\begin{defin}\label{def:lifting} 
An \emph{$\wt{S}$-lifting} of $X/S$ is a flat morphism $\wt{X} \to \wt{S}$ together with an isomorphism $\wt{X} \times_{\wt{S}} S \isom X$.  
\end{defin}
We say that $X/S$ admits an $\wt{S}$-lifting or is $\wt{S}$-liftable if there exists a lifting $\wt{X}/\wt{S}$ as described above.  In particular, we say that $X/k$ lifts to $W_2(k)$ or is \emph{$W_2(k)$-liftable} if it admits a lifting $\wt{X}/W_2(k)$.  Moreover, we say that it lifts compatibly with Frobenius or is Frobenius liftable (abbr.\ $F$-liftable) if there exists a lifting $\wt{X}$ of $X$ together with a morphism $\wt{F} : \wt{X} \to \wt{X}$ over the Frobenius of $W_2(k)$ which restricts to $F_X : X \to X$, i.e., such that the following diagram is commutative:
\begin{displaymath}
    \xymatrix{ 
    		& X\ar[rr]^{F_X}\ar[ld]_i\ar[dd] & & X\ar[ld]_{i} \ar[dd] 	\\
       		\wt{X} \ar[rr]^(.55){\wt{F}}\ar[dd] & & \wt{X} \ar[dd] & \\ 
		& \Spec(k)\ar[ld]_h\ar[rr]^(.45){F_k} & & \Spec(k) \ar[ld]_{h} \\
		\Spec(W_2(k)) \ar[rr]^{\sigma} & & \Spec(W_2(k)).
		}
\end{displaymath}

For a perfect field $k$, we can reformulate Frobenius liftability in terms of the relative Frobenius morphism $F_{X/k} \colon X \to X \times_{k,F_k} \Spec(k)$ and its lifting $\wt{F}_{X/k} \colon \wt{X} \to \wt{X} \times_{W_2(k),\sigma} \Spec(W_2(k))$.

In the affine case, we sometimes identify liftings of $k$-schemes (resp. their Frobenius morphisms) with liftings of their associated $k$-algebras (resp. their $p$-th power mappings).

\subsubsection{Derived categories and derived functors}

Throughout the paper by $D_{\QCoh}(X)$ we denote the derived category of quasi-coherent sheaves on a scheme $X$.  We add $+$ (respectively $-$) in the superscript in order to indicate that we consider the category of bounded below (respectively above) complexes.  Moreover, by $D^-_{\Coh}(X)$ we denote the full subcategory of $D_{\QCoh}^-(X)$ consisting of complexes whose cohomology sheaves are coherent. Again, based on our assumptions, one can prove that any object in $D_{\Coh}^-(X)$ is quasi-isomorphic to a complex of coherent sheaves.   

\subsection{Structure of the paper}

The paper is organized as follows.  In \cref{sec:preliminaries} we give a few preliminary results concerning commutative algebra and homological algebra.  

In \cref{sec:witt_liftability} we investigate $W_2(k)$-liftability of schemes.  In particular, we prove that under suitable assumptions Witt vector liftability is a constructible property (see \cref{thm:constructible}) and give a construction of a $W_2(k)$-lifting of a Frobenius split scheme (see \cref{thm:construction_split}). 

In \cref{sec:frobenius_liftability} we address the problem of Frobenius liftability of singularities.  We derive a computational criterion for Frobenius liftability of complete intersection schemes (see \cref{lem:complete_intersections_criterion}) and consequently apply it to ordinary double points (see \cref{thm:cones_bhatt}) and canonical surface singularities (see \cref{thm:canonical_surface_frobenius}).  \cref{sec:relation_f-singularities} is devoted to comparing $W_2(k)$-liftability and $F$-liftability with classical notions of $F$-regularity, $F$-rationality and $F$-purity.   In \cref{sec:cones} we recall the classical approach to deformation of cones with an emphasis on $W_2(k)$-liftability.  Moreover, we extend the standard results with a study of Frobenius liftability. 

Due to their technical nature deformation theory considerations are deferred to \cref{appendix:deformation}.   We recall a few results concerning deformation functors and present a thorough description of functoriality of obstruction classes expressed in terms of cotangent complex. Moreover, we present a handful of remarks concerning deformations of affine schemes.


\section{Preliminaries}\label{sec:preliminaries}


\subsection{Basic commutative algebra, flatness criteria}

We recall basic facts concerning flatness of $W_2(k)$-modules and algebras.  Under standard bijection $W_2(k) \isom k^2$ the element $p \in W_2(k)$ corresponds to $(0,1) \in k^2$.

\begin{lemma}\label{lem:lifting}
A $W_2(k)$-module $M$ is flat if and only if the annihilator $Ann_M( p) = \{m \in M: pm = 0\}$ of $p$ in $M$ is equal to $pM$, i.e., the natural mapping $M/pM \to pM$ given by $[m] \mapsto pm$ is injective.
\end{lemma}
\begin{proof}
We firstly prove that $pW_2(k) \otimes_{W_2(k)} M$ is isomorphic to $M/pM$.  This follows from right exactness of the functor $- \otimes M$ and a resolution of $pW_2(k)$:
\begin{displaymath}
    \xymatrix{
        W_2(k) \ar[r]^{p \cdot} & W_2(k) \ar[r]^{p \cdot} & pW_2(k) \ar[r] & 0. }
\end{displaymath}
By a well-known criterion stating that $M$ is flat over $W_2(k)$ if and only if $- \otimes_{W_2(k)} M$ preserves any injection of ideal of $W_2(k)$ (see \cite{atiyah_macdonald}, Exercise 2.24), the flatness of $M$ is now equivalent to the injectivity of the multiplication by $p$ mapping $pW_2(k) \otimes_{W_2(k)} M \to M$.  By the isomorphism above, this mapping is the same as $m_p : M/pM \to M$ given by $[m] \mapsto pm$.  The kernel of $m_p$ is exactly $Ann_M(p)/pM \subset M/pM$ and therefore injectivity is equivalent to the hypothesis of the lemma.
\end{proof}

As a simple corollary we obtain:

\begin{cor}\label{cor:lift_criterion}
A ring $B/W_2(k)$ is a flat lifting of $A/k$ if and only if $B/pB \isom A$ and $pB = Ann_B(p)$. Moreover, for any flat lifting $B/W_2(k)$ of $A$ the quotient $B/J$ is a flat lifting of $A/I$ if and only if $J$ is a lifting of $I$ and $(p) \cap J = pJ$.
\end{cor}
\begin{proof}
The first part follows directly from \cref{lem:lifting}. To prove the second part, we apply the following sequence of equivalences:
\begin{align}
Ann_{B/J}{p} = p \cdot B/J \Leftrightarrow (J : (p)) = (J+p) \Leftrightarrow p(J : (p)) = p(J+p) \Leftrightarrow (p) \cap (J) = pJ,
\end{align}
where the middle one follows from the inclusions $(J : (p)) \supset (p) \subset (J+p)$ and the injectivity of the mapping $B/p \to pB$.
\end{proof}

\begin{lemma}\label{lem:reg_seq}
Suppose $B/W_2(k)$ is a flat lifting of $B/p$.  Then, any lift $f \in B$ of a non-zerodivisor $\overline{f} \in B/pB$ is a non-zerodivisor.  Moreover, $B/fB$ is a $W_2(k)$-lifting of $B/(p,\overline{f})B$.
\end{lemma}
\begin{proof}
Assume that $g \in B$ satisfies the equality $gf = 0$.  By reducing mod $p$ we see that $\overline{g}\cdot \overline{f} = 0$ and therefore $g = ph$ for some $h \in B$.  Again, using the assumption and the isomorphism $B/pB \isom pB$ (see \cref{lem:lifting}), we see that $\overline{h} \cdot \overline{f} = 0$ and consequently $h = ph'$.  This implies that $g = p^2h' = 0$ yielding the first part of the lemma.  For the second part, we consider the flat resolution of $B/fB$ given by $0 \ra B \ra B \ra B/fB \ra 0$.  By reducing mod $p$ we see that $\Tor^{W_2(k)}_1(B/fB,k) = 0$ and therefore $B/fB$ is $W_2(k)$-flat (see \cite[\href{http://stacks.math.columbia.edu/tag/051C}{Tag 051C}]{stacks-project}).  Alternatively, we can prove that $(f) \cap (p) = (pf)$ and use the above \cref{cor:lift_criterion}.
\end{proof}

Consequently we get the following result.

\begin{lemma}\label{lemma:def_ci}
Let $I = (f_1,\ldots,f_m)$ be an ideal in $k[x_1,\ldots,x_n]$.  Every $W_2(k)$-lifting of $k[x_1,\ldots,x_n]/I$ is given by $W_2(k)[x_1,\ldots,x_n]/J$, where $J$ is generated a sequence of lifts of $\{f_i\}$.  Conversely, if $(f_1,\ldots,f_m)$ is a regular sequence then for every ideal $J$ generated by the sequence of lifts of $\{f_i\}$ the ring $W_2(k)[x_1,\ldots,x_n]/J$ is a $W_2(k)$-lifting.
\end{lemma}
\begin{proof}
For the first part see \cite[{Theorem 10.1}]{hartshorne_deformation}.  For the converse, apply the above \cref{lem:reg_seq}, inductively.
\end{proof}


\subsection{Base change for higher direct images}

We now present a few results concerning derived inverse images and base change for unbounded complexes.

\begin{defin}[Morphism of finite Tor dimension]
We say that a morphism of schemes $f : X \to Y$ is of \emph{finite Tor dimension} if $\cO_X$ is a $f^{-1}\cO_Y$-module of finite Tor dimension. 
\end{defin}

For any morphism $f : X \to Y$ of finite Tor dimension the structural sheaf $\cO_X$ is in fact quasi-isomorphic to a bounded complex of flat $f^{-1}\cO_Y$-modules (see  \cite[\href{http://stacks.math.columbia.edu/tag/08CI}{Tag 08CI}]{stacks-project}).

\begin{remark}
In the case of morphisms of quasi-projective schemes over a field by \cite[II.1.2]{fulton_macpherson} we know that the following morphisms are of finite Tor dimension:
\begin{enumerate}[a)] 
\item morphism with a smooth target, 
\item flat morphisms,
\item regular and Koszul closed immersions. 
\end{enumerate}
\end{remark}

Due to the lack of an adequate reference, we give a detailed proof of the following lemma.

\begin{lemma}[Spectral sequence for derived pullback]\label{lem:spectral_seq}
Suppose $f : X \to Y$ is a morphism of finite Tor dimension.  Then, for any $\cF^{\bullet} \in D(\cO_X)$ there exists a spectral sequence:
\[
E^{pq}_2 = L^pf^*\left(\cH^q(\cF^{\bullet})\right) \Rightarrow L^{p+q}f^*\cF^{\bullet}.
\] 
\end{lemma}
\begin{proof}
This is an application of the spectral sequence of a double complex.  Indeed, by \cite[\href{http://stacks.math.columbia.edu/tag/08DE}{Tag 08DE}]{stacks-project} we see that the total derived pullback $Lf^*\cF^\bullet$ is given by the totalisation of the double complex $f^{-1}\cF^\bullet \otimes_{f^{-1}\cO_Y} \cE^\bullet$ where $\cE^\bullet$ is a bounded $f^{-1}\cO_Y$-flat resolution of $\cO_X$ existing by the assumption on $f$.  By the classical result there exist a spectral sequence: 
\[
E^{pq}_2 = \cH^p\cH^q(f^{-1}\cF^\bullet \otimes_{f^{-1}\cO_Y} \cE^\bullet) \Rightarrow \cH^{p+q}\mathrm{Tot}(f^{-1}\cF^\bullet \otimes_{f^{-1}\cO_Y} \cE^\bullet),
\]
whose convergence follows from the fact that $f^{-1}\cF^\bullet \otimes_{f^{-1}\cO_Y} \cE^\bullet$ is a double complex supported in the bounded horizontal strip $\{(i,j) : j \in \{k : \cE^k \neq 0\}\}$.  The sheaves $\cH^{p+q}\mathrm{Tot}(f^{-1}\cF^\bullet \otimes_{f^{-1}\cO_Y} \cE^\bullet)$ are easily identified with $L^{p+q}f^*\cF^{\bullet}$.  Moreover, by the exactness of the functor $f^{-1}$ and flatness of $\cE^\bullet$ (in fact, we use \cite[Tag 08DE]{stacks-project} again), we see that 
\[
\cH^p\cH^q(f^{-1}\cF^\bullet \otimes_{f^{-1}\cO_Y} \cE^\bullet) \isom \cH^p(f^{-1}\cH^q(\cF^\bullet) \otimes_{f^{-1}\cO_Y} \cE^{\bullet}) \isom L^pf^*\cH^q(\cF^\bullet).
\] 
This gives the desired result.
\end{proof}

\begin{lemma}\label{lem:cohomology_base_change}
Let $f : X \to S$ be a locally of finite type morphism of schemes, and $\cc{E}$ a complex in $D_{\QCoh}(X)$.  Then for every cartesian diagram:
\begin{displaymath}
\xymatrix{
	X_T \ar[d]^{f_T}\ar[r]^{i_f} & X \ar[d]^f \\
	T \ar[r]^i & S,}
\end{displaymath} 
the natural base change $\phi : Li^*Rf_*\cc{E} \isomto Rf_{T*}Li_f^*\cc{E}$ is an isomorphism.  
\end{lemma}
\begin{proof}
For the proof we refer to \cite[\href{http://stacks.math.columbia.edu/tag/0A1D}{Tag 0A1D}]{stacks-project}. 
\end{proof}

\begin{lemma}\label{lem:derived_dual}
Suppose $f : X \to Y$ is a morphism of schemes.  Then, for every $\cF^\bullet \in D^-_{\Coh}(Y)$ there exists a natural isomorphism $Lf^*\RcHom(\cF^\bullet,\cO_Y) \isom \RcHom(Lf^*\cF^\bullet,\cO_X)$.
\end{lemma}
\begin{proof}
By \cite[\href{http://stacks.math.columbia.edu/tag/08I3}{Tag 08I3}]{stacks-project} we obtain a mapping: 
\[
Lf^*\RcHom(\cF^\bullet,\cO_Y) \to \RcHom(Lf^*\cF^\bullet,\cO_X).
\]  
In order to prove that it is an isomorphism we can work locally and therefore we may assume that $\Cot_{X/S}$ is resolved as a complex of locally free sheaves.  In this case, the result follows from the definition of $Lf^*$ and the fact that $\RcHom(-,\cO_X)$ might be computed by a locally free resolution of the first argument.
\end{proof}

\subsection{A primer in deformation theory} \label{sec:deformation_theory}

Due to unoriginal and tediously technical nature of the considerations given in the corresponding section we defer it to \cref{appendix:deformation}.  The reader familiar with 1) obstruction theories defined in terms of cotangent complex, 2) the notion of deformation functor and its basic properties, may comfortably proceed with the actual content of the paper.  In the upcoming considerations we freely use the above notions by referring to \cref{appendix:deformation_cotangent} and \cref{appendix:deformation_functors}, respectively. 


\section{$W_2(k)$-liftability} \label{sec:witt_liftability}

In this section we begin the main considerations of this work.  Firstly, we present general facts concerning $W_2(k)$-liftability and give a few examples.  Subsequently, we prove that the locus of $W_2(k)$-liftable schemes in a flat family $X \to S$ is constructible.  Finally, we reprove constructively the classical result (see \cref{thm:lift_split}) that any Frobenius-split variety admits a $W_2(k)$-lifting.  We begin with functoriality of obstructions to existence of $W_2(k)$-liftings.

\subsection{General results on obstruction to liftability}

\begin{lemma}\label{lem:func_w2k}
For any scheme $X/k$ there exists an obstruction $\sigma_X \in \Ext^2(\Cot_{X/k},\cc{O}_X)$ whose vanishing is sufficient and necessary for existence of $W_2(k)$-lifting.  The obstruction is functorial, i.e., for any $k$-scheme morphism $g : X \to Y$ the following diagrams in $D_{\QCoh}(X)$ and $D_{\QCoh}(Y)$ are commutative:
\begin{displaymath}
    \xymatrix{
       Lg^*\Cot_{Y/k} \ar[r]^{Lg^*\sigma_{Y}}\ar[d]^{dg} & Lg^*\cc{O}_{Y}[2] \ar[d]^{\isom} 	& & \Cot_{Y/k} \ar[r]^{\sigma_{Y}}\ar[d]^{dg} & \cc{O}_{Y}[2] \ar[d]^{g^{\#}[2]} \\
        \Cot_{X/k} \ar[r]^-{\sigma_X} & \cc{O}_X[2], 								& & Rg_*\Cot_{X/k} \ar[r]^-{Rg_*\sigma_X} & Rg_*\cc{O}_X[2].}  
\end{displaymath}
In particular $W_2(k)$-liftability descends along finite surjective maps of degree prime to the characteristic of $k$.
\end{lemma}
\begin{proof}
This is a direct consequence of \cref{lem:obstruction_schemes_func}.  The final remark follows from existence of  the trace maps splitting $g^\#$.
\end{proof}

As corollary we obtain a well-known result that every Frobenius split variety lifts.

\begin{prop}\label{thm:lift_split}
If $X/k$ is a Frobenius split scheme then $X^{(1)}/k$ lifts to $W_2(k)$. 
\end{prop}
\begin{proof}
See \cite[p. 164]{illusie_frobenius} or \cite[Corollary 9.2]{joshi_exotic} for the case of smooth schemes.  We reproduce general proof by Bhargav Bhatt given in \cite[Proposition 8.4]{langer}.  The idea is to use the functoriality of obstructions for the relative Frobenius $F_{X/k} : X \to X^{(1)}$.  Namely, we have the following commutative diagram:
\begin{displaymath}
    \xymatrix{
        \Cot_{X^{(1)}/k} \ar[rr]^{\sigma_{X^{(1)}}}\ar[d]^{dF_{X/k}} & & \cc{O}_{X^{(1)}}[2] \ar[d]\\
        F_{X/k*}\Cot_{X/k} \ar[rr]^-{F_{X/k*}\sigma_X} & & F_{X/k*}\cc{O}_X[2]. \ar@/^-1.0pc/[u] }
\end{displaymath}
The differential $dF_{X/k} = 0$ and therefore by existence of splitting $\sigma_{\Frob{X}} = 0$.  Note that in case of varieties over a perfect field the $W_2(k)$-liftability of $X^{(1)}$ is in fact equivalent to the liftability of $X$.
\end{proof}

\subsection{Example of non $W_2(k)$-liftable scheme, $p$-neighbourhoods of smooth quadrics}\label{sec:quadric_example}

We proceed to an example of $0$-dimensional scheme which is not $W_2(k)$-liftable.  In this case we give a direct computational proof which is afterwards used in the considerations concerning Frobenius liftability of ordinary double points.  We shall prove that the Artinian local $k$-algebras:

{\scriptsize
\begin{align*}
(A_{2n},\mathfrak{m}_{2n}) & = \left(\bigslant{k[x_1,\ldots,x_{2n}]}{(x_1x_2 + \ldots + x_{2n-1}x_{2n},x_1^p,\ldots,x_{2n}^p)},(x_1,\ldots,x_{2n})\right) \\
(A_{2n+1},\mathfrak{m}_{2n+1}) & = \left(\bigslant{k[x_1,\ldots,x_{2n+1}]}{(x_1x_2 + \ldots + x_{2n-1}x_{2n} + x_{2n+1}^2,x_1^p,\ldots,x_{2n+1}^p)},(x_1,\ldots,x_{2n+1}\right)
\end{align*}
}
are non-liftable to $W_2(k)$ for $n \geq 3$.  Moreover, as suggested in \cite{OgusBerthelot}[page 3.3], their maximal ideals do not admit PD-structure. 


\subsubsection{Non-liftability to $W_2(k)$}

\begin{prop}\label{prop:non_liftability_quadric}
For any $N \geq 5$ the local algebra $A_N$ does not admit a $W_2(k)$-lifting.
\end{prop} 
\begin{proof}
Let $f_{2n}$ denote the polynomial $x_1x_2 + \ldots + x_{2n-1}x_{2n}$ and $f_{2n+1}$ the polynomial $x_1x_2 + \ldots + x_{2n-1}x_{2n} + x_{2n+1}^2$.  We begin with the case of even number of variables $N = 2n$.  By \cref{lemma:def_ci}, it suffices to show that for any choice of liftings $\wt{f_{2n}},\wt{x_1^p},\ldots,\wt{x_{2n}^p}$ of $f_{2n},x_1^p,\ldots,x_{2n}^p$, the ring 
\[
\wt{A} = W_2[x_1,\ldots,x_{2n}]/(\wt{f_{2n}},\wt{x_1^p},\ldots,\wt{x_{2n}^p})
\] is not flat over $W_2(k)$.  By means of \cref{lem:reg_seq} we see that 
\[
\wt{C} = W_2[x_1,\ldots,x_{2n}]/(\wt{x_1^p},\ldots,\wt{x_{2n}^p})
\] is a flat $W_2(k)$-lifting of $C \mydef \wt{C} \otimes_{W_2(k)} k = k[x_1,\ldots,x_{2n}]/(x_1^p,\ldots,x_{2n}^p)$.  Consequently it suffices to show that the canonical homomorphism $Ann_{\wt{C}}(\wt{f_{2n}}) \to Ann_{C}(f_{2n})$ coming from the snake lemma applied to the commutative diagram:
\begin{displaymath}
    \xymatrix{
        & Ann_{\wt{C}}(\wt{f_{2n}}) \ar[r]\ar[d] & Ann_{C}(f_{2n})\ar[d] & \\
        C \ar[r]^{\cdot p}\ar[d]^{\cdot f_{2n}} & \wt{C} \ar[r]\ar[d]^{\cdot \wt{f_{2n}}} & C\ar[d]\ar[r] & 0 \\
        C \ar[r]^{\cdot p}\ar[d] & \wt{C} \ar[r]\ar[d] & \wt{C} \ar[d]\ar[r] & 0 \\
        A \ar[r]^{\cdot p} & \wt{A} \ar[r] & A, & }
\end{displaymath}
is not surjective, i.e., there exists an element $h$ in the annihilator of $f_{2n} \in C$ which does not lift to an element in the annihilator of $\wt{f_{2n}}$ in $\wt{C}$.  We claim that $h = f_{2n}^{p-1}$ is the right choice.  For the sake of contradiction, suppose that there exists a lifting $\wt{f_{2n}}^{p-1} + pg$ such that $(\wt{f_{2n}}^{p-1} + pg)\wt{f_{2n}} = \wt{f_{2n}}^p + pgf_{2n} = 0 \in \wt{C}$.
By direct computation we see that
\begin{align*}
\wt{f_{2n}}^p = (x_1x_2 + \ldots + x_{2n-1}x_{2n} + pu)^p & = (x_1x_2 + \ldots + x_{2n-1}x_{2n})^p \\
&= x_1^px_2^p + \ldots + x_{2n-1}^px_{2n}^p + pP_{f_{2n}} = pP_{f_{2n}}
\end{align*}
for $P_{f_{2n}}$ defined as
\[
P_{f_{2n}} =  \sum_{\substack{{i_1+\ldots+i_n = p} \\ i_1,\ldots,i_n \neq p}} \frac{(p-1)!}{i_1! \cdots i_n!}(x_1x_2)^{i_1} \cdots (x_{2n-1}x_{2n})^{i_n} \in \bigslant{\wt{C}}{p\wt{C}} \isom k[x_1,\ldots,x_{2n}]/(x_1^p,\ldots,x_{2n}^p).
\]
Therefore, we obtain that $p(P_{f_{2n}} + gf_{2n}) = 0$ in $\wt{C}$, i.e., by \cref{cor:lift_criterion} the polynomial $P_{f_{2n}}$ belongs to the ideal $(f_{2n},x_1^p,\ldots,x_{2n}^p)$.  By specializing $x_i = 0$ for $i>6$ we see that it suffices to derive the contradiction for $n = 3$.
We prove that even the element $(x_5x_6)^{p-2}P_{f_6}$ does not belong to $(f_3,x_1^p,\ldots,x_6^p)$.  We begin with the computation:
{\scriptsize
\begin{align*}
(x_5x_6)^{p-2}P_{f_6} & = (x_5x_6)^{p-2}\sum_{\substack{{i+j+k = p} \\ i,j,k \neq p}} \frac{(p-1)!}{i!j!k!}(x_1x_2)^i (x_3x_4)^j(x_5x_6)^k \\
 & = (x_5x_6)^{p-2} \sum_{\substack{{i+j = p} \\ i,j \neq p}} \frac{(p-1)!}{i!j!}(-1)^i(x_3x_4 + x_5x_6)^i (x_3x_4)^j \\
 & + (x_5x_6)^{p-1} \sum_{i+j = p-1} \frac{(p-1)!}{i!j!}(-1)^i(x_3x_4 + x_5x_6)^i (x_3x_4)^j \\
 & = (x_5x_6)^{p-1} \sum_{\substack{{i+j = p} \\ i,j \neq p}} \frac{(p-1)!}{i!j!}(-1)^i i(x_3x_4)^{i-1}(x_3x_4)^j + (x_5x_6)^{p-1} \sum_{i+j = p-1} \frac{(p-1)!}{i!j!}(-1)^i(x_3x_4)^{i+j} \\
 & = (x_3x_4x_5x_6)^{p-1} \underbrace{\sum_{\substack{{k+j = p-1} \\ k \neq p-1}} \frac{(p-1)!}{k!j!}(-1)^k}_{=-1} + (x_3x_4x_5x_6)^{p-1} \underbrace{\sum_{i+j = p-1} \frac{(p-1)!}{i!j!}(-1)^i}_{=0} \\
 & = -(x_3x_4x_5x_6)^{p-1} \pmod{f_6,x_1^p,\ldots,x_6^p}
\end{align*}
}
Therefore, we are left to show that $(x_3x_4x_5x_6)^{p-1} \not\in (f_6,x_1^p,\ldots,x_6^p)$.  This follows from the following:
\begin{align*}
(x_1x_2 + x_3x_4 + x_5x_6) \left(\sum_{i,j} g_{i,j} \cdot x_1^ix_2^j \right) & = (x_3x_4x_5x_6)^{p-1}\pmod{x_1^p,\ldots,x_6^p} \\
\left(\sum_{i,j} g_{i,j} \cdot x_1^{i+1}x_2^{j+1} \right) + \left(\sum_{i,j} g_{i,j}(x_3x_4 + x_5x_6) \cdot x_1^ix_2^j \right) & = (x_3x_4x_5x_6)^{p-1}\pmod{x_1^p,\ldots,x_6^p} \\
\left(\sum_{i,j} \left(g_{i-1,j-1} + g_{i,j}(x_3x_4 + x_5x_6) \right) \cdot x_1^ix_2^j \right) & = (x_3x_4x_5x_6)^{p-1}\pmod{x_1^p,\ldots,x_6^p}.
\end{align*}

As $\bigslant{k[x_1,\ldots,x_6]}{(x_1^p,\ldots,x_6^p)}$ is a free $\bigslant{k[x_3,\ldots,x_6]}{(x_3^p,\ldots,x_6^p)}$-module on generators $x_1^ix_2^j$ we may compare coefficients to obtain:
\begin{align*}
g_{0,0} (x_3x_4 + x_5x_6) = (x_3x_4x_5x_6)^{p-1}\pmod{x_3^p,\ldots,x_6^p} \\
g_{i,i} = -g_{i+1,i+1} (x_3x_4 + x_5x_6) \pmod{x_3^p,\ldots,x_6^p}.
\end{align*}
After using the second relation inductively we obtain: 
\begin{align*}
0 = g_{p-1,p-1} (x_3x_4 + x_5x_6)^{p} & = g_{p-1,p-1} (x_3x_4 + x_5x_6)^{p-1}  (x_3x_4 + x_5x_6) \\
& = g_{0,0}  (x_3x_4 + x_5x_6) \\
& = (x_3x_4x_5x_6)^{p-1}\pmod{x_3^p,\ldots,x_6^p},
\end{align*}
which is a contradiction.
\end{proof}

The proof for odd number $N = 2n+1 \geq 7$ of variables is analogously reduced to the case of $6$ variables by additional specialisation $x_{2n+1} = 0$. The case $N = 5$ can be treated easily by an explicit computation following the lines of the above.

\begin{remark}
We have implemented $\mathsf{Macaulay 2}$ procedure checking whether a give affine scheme is $W_2(k)$-liftable.  The code is given at \href{http://www.mimuw.edu.pl/~mez/Macaulay2}.
\end{remark}

\subsubsection{Non-existence of a PD-structure on $(A,\ideal{m})$}

We now prove that the ideal $\ideal{m}$ does not admit PD-structure in spite of satisfying the necessary condition $\ideal{m}^{[p]} = 0$.
We precede the actual result with a necessary definition.

\begin{defin}[PD-structure]
Let $A$ be a commutative ring and let $I$ be an ideal of $A$. A \emph{divided power structure} is a collection of mappings $\gamma_i : I \to A$ for $i \in \bb{N}$ satisfying the following properties:
\begin{enumerate}[i)]
\item $\gamma_0(x) = 1$, $\gamma_1(x) = x$ for any $x \in I$ and $\gamma_n(x) \in I$ for any $n \geq 1$,
\item $\gamma_n(x)\gamma_m(x) = \frac{(n+m)!}{n!m!}\gamma_{n+m}(x)$,
\item $\gamma_n(\gamma_m(x)) = \frac{(nm)!}{(m!)^n n!}\gamma_{nm}(x)$,
\item $\gamma_n(ax) = a^n\gamma_n(x)$ for any $x \in I$ and $a \in A$,
\item $\gamma_n(x+y) = \sum_{i+j = n} \gamma_i(x)\gamma_j(y)$.
\end{enumerate}
\end{defin}

The above notion is a necessary tool for dealing with de Rham and crystalline cohomology in characteristic $p$.  

Using ii) we see that $p! \gamma_p(x) = x^p$, which indeed implies that any ideal $I$ of an $\mathbb{F}_p$-algebra $A$ admitting a PD-structure satisfies the condition $I^{[p]} = 0$.

\begin{prop}
For $N \geq 5$ the local algebras $(A_N,\ideal{m}_N)$ defined above do not admit PD-structures on the maximal ideals $\ideal{m}_N$.
\end{prop}
\begin{proof}
We focus on the case of even $N = 2n$.  For the sake of contradiction, we assume an appropriate system of mappings $\gamma_i$ exists and analyse the element $\gamma_p(f_{2n}) = \gamma_p(0) = 0$.  
\begin{align*}
0 = \gamma_p(f_{2n}) & = \sum_{i_1+\ldots+i_n = p} \gamma_{i_1}(x_1x_2) \cdots \gamma_{i_n}(x_{2n-1}x_{2n}) \\
& = \sum_{\substack{{i_1+\ldots+i_n = p} \\ i_1,\ldots,i_n \neq p}} \frac{(x_1x_2)^{i_1}}{i_1!} \cdots \frac{(x_{2n-1}x_{2n})^{i_n}}{i_n!} = \frac{1}{(p-1)!} P_{f_{2n}} \pmod{f_{2n},x_1^p,\ldots,x_{2n}^p}.
\end{align*}
This means that $P_{f_{2n}} \in (f_{2n},x_1^p,\ldots,x_{2n}^p)$ which gives a contradiction as proven above.
\end{proof}

\subsection{$W_2(k)$-liftability in families}\label{sec:families}

We now proceed to the proof that under suitable conditions Witt vector liftability is a constructible property, i.e., for a locally finite type family $f : X \to S$ the set of closed points $s \in S$ such that the fibre $X_s$ is $W_2(k(s))$-liftable is constructible.  We begin with a definition encompassing the properties of a morphism necessary to prove constructibility.

\begin{defin}
We say that a morphism $f : X \to S$ is \emph{strongly equicohomological} if it is locally of finite type and flat, and for any $i \in \{0,\ldots,\dim S+2\}$ the higher direct image $R^if_*\RcHom(\Cot_{X/S},\cO_X)$ is a flat $\cO_S$-module.
\end{defin}

Strongly equicohomological families satisfy the following crucial property motivated by the application for deformation theory.

\begin{lemma}\label{lem:tamely_base_change}
Let $f : X \to S$ be a strongly equicohomological morphism with a smooth target $S$.  For any morphism $T \to S$ the sheaf $R^2 f_*\RcHom(\Cot_{X/S},\cc{O}_X)$ satisfies base change property, i.e., for any cartesian diagram:
\begin{displaymath}
\xymatrix{
	X_T \ar[d]^{f_T}\ar[r]^{j} & X \ar[d]^f \\
	T \ar[r]^i & S.}
\end{displaymath} 
there exists a natural isomorphism $i^*R^2 f_*\RcHom(\Cot_{X/S},\cc{O}_X) \isom R^2f_{T*}\RcHom(\Cot_{X_T/T},\cO_{X_T})$. 
\end{lemma}
\begin{proof}
By \cref{lem:cohomology_base_change} there exists a natural quasi-isomorphism:
\begin{align}
Li^*Rf_*\RcHom(\Cot_{X/S},\cc{O}_X) \to Rf_{T*}Lj^*\RcHom(\Cot_{X/S},\cc{O}_X).\label{qi:one}
\end{align}  
By \cref{lem:derived_dual} and the base change property of cotangent complex we obtain an isomorphism:
\[
Lj^*\RcHom(\Cot_{X/S},\cc{O}_X) \isom \RcHom(Lj^*\Cot_{X/S},Lj^*\cc{O}_X)\isom \RcHom(\Cot_{X_T/T},\cc{O}_{X_T}),
\]
which implies that the right hand side of (\ref{qi:one}) is isomorphic to $Rf_{T*}\RcHom(\Cot_{X_T/T},\cc{O}_{X_T})$.  By taking cohomology, for any $k \in \ZZ$ we obtain an isomorphism:
\[
L^ki^*Rf_*\RcHom(\Cot_{X/S},\cc{O}_X) \to R^kf_{T*}\RcHom(\Cot_{X_T/T},\cc{O}_{X_T}).
\]
By \cref{lem:spectral_seq} there exists a convergent spectral sequence: 
\[
E^{pq}_2 = L^pi^*\cH^q(Rf_*\RcHom(\Cot_{X/S},\cc{O}_X)) \Rightarrow L^{p+q}i^*Rf_*\RcHom(\Cot_{X/S},\cc{O}_X).
\]
The terms on $E_2$ page are isomorphic to $L^pi^*R^qf_*\RcHom(\Cot_{X/S},\cc{O}_X))$ and therefore by the assumption on being strongly equicohomological $E^{pq}_2 = 0$ for $(p,q) \in \ZZ_{<0} \times \{0,\ldots,\dim S + 2\}$ and $p>0$.  Moreover, by the smoothness of $S$ we see that $E^{pq}_2 = 0$ for $p < -\dim S$ and consequently $E^{-p,p+2}_r \isom E^{-p,p+2}_{r+1}$ for any $p \in \ZZ$ and $r \geq 2$.  This means that the natural filtration induced on $L^{2}i^*Rf_*\RcHom(\Cot_{X/S},\cO_X)$ by the spectral sequence consists of a single term $i^*R^2f_*\RcHom(\Cot_{X/S},\cO_X)$.  This finishes the proof.  
\end{proof}

As a corollary we obtain:

\begin{lemma}\label{lem:weakly_equisingular}
Let $f : X \to S$ be a strongly equicohomological morphism with a smooth affine target.  Then, the set of closed points $s \in S$ such that $X_s$ is $W_2(k(s))$-liftable is closed.
\end{lemma}
\begin{proof}
By the assumptions on $S$ we see that there exists a $W_2(k)$-lifting $\wt{S}$ of $S$.  Therefore, there is a relative obstruction class $\sigma_f \in \Ext^2(\Cot_{X/S},f^*(p\cc{O}_{\wt{S}})) = \Ext^2(\Cot_{X/S}, \cc{O}_X)$ which vanishes if and only if there exists a flat $\wt{S}$-scheme $\wt{X}$ fitting into the cartesian diagram:
\begin{displaymath}
\xymatrix{
	X \ar[d]^{f}\ar@{-->}[r] & \wt{X}\ar@{-->}[d] \\
	S \ar[r]\ar[d] & \wt{S} \ar[d] \\
	\Spec(k) \ar[r] & \Spec(W_2(k)).}
\end{displaymath} 

By the formal smoothness of $\wt{S} \to \Spec(W_2(k))$, any point $s \in S$ can be lifted to a mapping $\wt{i} : \Spec(W_2(k(s)) \to \wt{S}$.  We denote by $i_s : X_s \to X$ the closed immersion defined by the diagram (with a cartesian square):
\begin{displaymath}
    \xymatrix{
         & X_s \ar[dl]^-{i_s}\ar[d]^-{f_s} &  \\
        X\ar[d]^-{f} & \Spec(k(s)) \ar[r]\ar[ld]^-{i} & \Spec(W_2(k(s)) \ar[dl]^-{\wt{i}} \\
        S \ar[r] & \wt{S}, &
        }  
\end{displaymath} 
and by $\sigma_s \in \Ext^2(\Cot_{X_s/s},\cc{O}_{X_s})$ the associated obstruction class to lifting $X_s/k(s)$ to $W_2(k(s))$.  
The $\Gamma(S,\cc{O}_S)$-module $\Ext^2(\Cot_{X/S},\cc{O}_X)$ is in fact a set of global sections of the sheaf 
\[
R^2 f_*\RcHom(\Cot_{X/S},\cc{O}_X),
\] 
and therefore by \cref{lem:tamely_base_change} we obtain a specialisation isomorphism:
\begin{align*}
\psi_s : \Ext^2(\Cot_{X/S},\cO_X) \otimes_{\cc{O}_S} k(s) & = i^*R^2 f_*\RcHom(\Cot_{X/S},\cc{O}_X) \\
& \isom R^2f_{s*} \RcHom(\Cot_{X_s/k(s)},\cc{O}_{X_s}) = \Ext^2(\Cot_{X_s/k(s)},\cO_{X_s}).
\end{align*}
By \cref{lem:obstruction_schemes_func} we see that obstructions $\sigma_f \in \Ext^2(\Cot_{X/S},\cO_X)$ and $\sigma_s \in \Ext^2(\Cot_{X_s/k(s)},\cO_{X_s})$ fit into commutative diagram:
\begin{displaymath}
    \xymatrix{
            Li_s^*\Cot_{X/S} \ar[r]^{i_s^*\sigma_f}\ar[d]^{di_s}	& Li_s^*\cO_X \ar[d] \\
        \Cot_{X_s/k(s)} \ar[r]^-{\sigma_s} 	& \cO_{X_s},}  
\end{displaymath}
and therefore the fibre $\psi_s([\sigma_f])$ of an obstruction class $\sigma_f \in \Ext^2(\Cot_{X/S},\cc{O}_X)$ is equal to the obstruction $\sigma_s \in \Ext^2(\Cot_{X_s/s},\cc{O}_{X_s})$. Note that here we implicitly use the fact that the second part of the isomorphism (\ref{qi:one}) in \cref{lem:tamely_base_change} is given by the differential $dj$.  Constructibility then follows from the fact that the zero set of a section of a flat module is constructible (in fact it is closed for a finitely generated module).
\end{proof}

In order to obtain a general statement for morphism which are not strongly equicohomological we need the following lemma preceded with a notational remark.  We say that a morphism $f : X \to S$ can be \emph{stratified into morphisms satisfying property $P$} if there exists a stratification of the target $\bigcup_i S_i = S$ into locally closed subschemes $S_i \subset S$ such that the base change $f_i : X \times_S S_i \to S_i$ satisfies $P$.

\begin{lemma}\label{lem:tamely_singular}
The following classes of morphism of schemes can be stratified into strongly equicohomological morphisms:
\begin{enumerate}[i)]
\item proper morphisms,
\item affine morphisms of finite type.
\end{enumerate}
\end{lemma}
\begin{proof}
We observe that stratification is a purely topological notion and therefore we may assume that $S$ is integral.  Now, it suffices to prove that for the classes in question there exists an open subset $U \in S$ such that $f_U : X \times_S U \to U$ is strongly equicohomological.  In both cases we apply the generic flatness to prove a more general statement that for any complex $\cE^\bullet \in D^+_{\Coh}(X)$ the higher direct images are generically $S$-flat.  Firstly, we deal with affine morphisms.  In this case, the higher direct images of any complex $\cE^\bullet \in D^+_{\Coh}(X)$ are computed by taking the pushforwards of the cohomology groups $\cH^j(\cE^\bullet)$.  Those are coherent sheaves on $X$ and are therefore generically $S$-flat.  Secondly, we treat the case of proper morphism.  Now, higher direct images $R^if_*\cE^\bullet$ are coherent $S$-modules and therefore are generically flat.
\end{proof}

We are now ready to state and prove the theorem.

\begin{thm}\label{thm:constructible}
Suppose $f: X \to S$ is a morphism which can be stratified into strongly equicohomological morphisms. Then, the set of closed points $s \in S$ such that the fiber $X_s \mydef X \times_S \Spec(k(s))$ lifts to $W_2(k(s))$ is constructible.
\end{thm}
\begin{proof}
Observe that in order to prove constructibility we may stratify the scheme $S$ into locally closed subsets.  Therefore, by the assumption on existence of strongly equicohomological stratification, we may assume that $ : X \to S$ is strongly equicohomological with a smooth affine target.  Then, the result follows from \cref{lem:weakly_equisingular}.
\end{proof}

As a direct corollary of \cref{lem:tamely_singular} and \cref{thm:constructible} we obtain:

\begin{cor}
For any proper or affine morphism $f : X \to S$ the set of closed points $s \in S$ such that the fibre $X_s$ lifts to $W_2(k(s))$ is constructible.
\end{cor}


\subsection{Functorial $W_2(k)$-lifting for Frobenius-split varieties}\label{sec:frobenius_split_witt}

We now reprove the classical result stating that any Frobenius-split scheme $X$ over $k$ is $W_2(k)$-liftable.  For the standard reference and Bhargav Bhatt's proof in singular case the reader is encouraged to see \cref{thm:lift_split}.  The advantage of our approach is that we provide the lifting constructively and functorially with respect to the splitting $\varphi : F_{X*}\cO_X \to \cO_X$, in fact as a certain subscheme of the Witt scheme $(X,W_2(\cO_X))$.

\subsubsection{Functorial setting}

We shall introduce our construction as a certain functor from the category described as follows.
By \emph{Frobenius split scheme} we mean a pair $(X,\varphi_X)$ of a scheme over $k$ and a Frobenius splitting $\varphi_X : F_*\cc{O}_X \to \cc{O}_X$.  Moreover, \emph{a morphism of Frobenius split schemes} $(X,\varphi_X)$ and $(Y,\varphi_Y)$ is a scheme morphism $\pi: X \to Y$ such that the splittings satisfy compatibility conditions expressed by commutativity of the diagram:

\begin{displaymath}
    \xymatrix{
        \cc{O}_Y \ar[r]^{\pi^{\#}}\ar@<1ex>[d]^{F^{\#}} & \pi_*\cc{O}_X \ar@<1ex>[d]^{\pi_*F^{\#}} \\
        F_*\cc{O}_Y \ar[r]\ar@<1ex>[u]^{\varphi_Y} & F_*\pi_*\cc{O}_X = \pi_*F_*\cc{O}_X \ar@<1ex>[u]^{\pi_*\varphi_X},}
\end{displaymath}

i.e., the relation $\pi^{\#} \circ \varphi_Y = (F_*\pi^{\#}) \circ (\pi_*\varphi_X)$ holds.  In case of an affine morphism corresponding to a homomorphism $f: A \to B$ this is equivalent to the equality $f\varphi_A = \varphi_B f$.  The notions described above allow us to introduce the category of Frobenius split varieties.

\begin{defin}
We define a category $\mathrm{Sch}^{split}_k$ of Frobenius split schemes over $k$ as a category with the set of objects consisting of all Frobenius split schemes $(X,\phi_X)$ and with the set of morphisms consisting of morphism of Frobenius split schemes $f : (X,\phi_X) \to (Y,\phi_Y)$.
\end{defin}

\subsubsection{Witt vectors}

We shall need the following explicit description of Witt vectors.

\begin{defin}[Witt vectors $W_2(A)$]
Suppose $A$ is a commutative ring.  We define the ring of Witt vectors $W_2(A)$ to be the set $A \times A$ equipped with addition $+_W$ and multiplication $\cdot_W$ given by the following formulas:
\begin{align}
(a_0,a_1) +_{W} (b_0,b_1) & = (a_0+b_0,a_1+b_1-P(a_0,b_0)) \\
(a_0,a_1) \cdot_{W} (b_0,b_1) & = (a_0b_0,a_0^pb_1+b_0^pa_1+pa_1b_1), 
\end{align}
where $P(a,b)$ is a polynomial $\frac{(a+b)^p - a^p - b^p}{p} \in \bb{Z}[a,b]$.
\end{defin}

The unit element of $W_2(A)$ is represented by $(1,0)$ and, in case of $\bb{F}_p$-rings, the prime number $p$ is represented by $(0,1)$.  The natural projection $(a_0,a_1) \mapsto a_0$ gives a ring homomorphism $\pi : W_2(A) \to A$.  In case of characteristic $p$ ring $A$, the ring $W_2(A)$ possesses a Frobenius endomorphism $\sigma_2 : W_2(A) \to W_2(A)$ given by the formula $(a_0,a_1) \mapsto (a_0^p,a_1^p)$, which is compatible with the Frobenius endomorphism $F: A \to A$, i.e., the identity $F\pi = \pi \sigma_2$ holds. 

\subsubsection{The construction of $W^\varphi_2(A)$}

We begin our considerations of explicit liftings of Frobenius split schemes by the affine case.  A simple consequence of \cref{cor:lift_criterion} is that a functorial construction $A \mapsto W_2(A)$ does not give a flat lifting (the ideal $pW_2(A)$ is not equal to $\mathrm{Ker}[W_2(A) \to A]$).  We therefore proceed in a different manner taking the Frobenius splitting into account.  Let $(A,\varphi)$ by a $k$-algebra together with a Frobenius splitting $\varphi$.  We claim that a flat lifting of $\Spec(A)$ over $W_2(k)$ is given by the following construction functorial with respect to natural mappings of $\bb{F}_p$ rings with Frobenius splitting, i.e., ring homomorphisms commuting with $p^{-1}$-linear splitting operators. 
\begin{defin}\label{def:twisted_witt}
The ring of twisted Witt vectors $W^\varphi_2(A)$ associated to a characteristic $p$ ring with a splitting $(A,\varphi)$ is a set $A \times A$ with addition and multiplication given respectively by the formulas:
\begin{align}
(a_0,a_1) \mplus (b_0,b_1) & \mydef (\id \times \varphi)[(a_0,a_1^p) +_W (b_0,b_1^p)] = (a_0+b_0,a_1+b_1-(\varphi \circ P)(a_0,b_0)) \\
(a_0,a_1) \mmul (b_0,b_1) & \mydef (\id \times \varphi)[(a_0,a_1^p) \cdot_W (b_0,b_1^p)] = (a_0b_0,a_0b_1+b_0a_1), 
\end{align}
which can also be described by the following diagrams:
\begin{displaymath}
    \xymatrix{
        (A \times A)^{\times 2} \ar[r]^{\mplus} \ar[d]^{(\id \times F)^{\times 2}} & A \times A & &   (A \times A)^{\times 2} \ar[r]^{\mmul} \ar[d]^{(\id \times F)^{\times 2}} & A \times A\\
        (A \times A)^{\times 2} \ar[r]^{+_W} & A \times A \ar[u]^{\id \times \varphi}     & &   (A \times A)^{\times 2} \ar[r]^{\cdot_W} & A \times A \ar[u]^{\id \times \varphi}. }
\end{displaymath}
\end{defin}
The verification that the above definition gives indeed a structure of an associative and commutative ring is a straightforward computation.  We are now ready to prove: 
\begin{thm}\label{thm:construction_split}
Any morphism of Frobenius-split algebras $f : (A,\varphi_A) \to (B,\varphi_B)$ induces a morphism of $f_{W_2} : W^{\varphi_A}_2(A) \to W^{\varphi_B}_2(B)$.  Moreover, the ring $W^\varphi_2(A)$ is a flat $W_2(k)$-lifting of $A$.  
\end{thm}
\begin{proof}
Firstly, we observe that our construction is functorial.  This follows directly from the formulas for operations and the condition $f \varphi_A = \varphi_B f$.

Consequently, we show that $W^\varphi_2(A)$ is flat over $W_2(k)$.  Clearly the projection onto the first factor gives a surjective ring homomorphism from $W^\varphi_2(A)$ to $A$ whose kernel $I$ is given by the ideal of elements of the form $(0,a_1) \in W^\varphi_2(A)$ for $a_1 \in A$, which is generated by a single element $q = (0,1)$.
Note that the above diagrams prove that in case of a perfect field $k$ with a Frobenius splitting $\gamma_k(\alpha) = \sqrt[p]{\alpha}$ the bijective mapping $\id \times F$ gives an isomorphism between the rings $W^{\gamma_k}_2(k)$ and $W_2(k)$.  This consequently means that natural homomorphism $W_2(k) \isom W^{\gamma_k}_2(k) \to W^\varphi_2(A)$ induced by the embedding $k \to A$ sends $p = (0,1) \in W_2(k)$ to $(0,1) \in W^\varphi_2(A)$ and therefore $q = p$.  This allows us to apply Corollary \ref{cor:lift_criterion} to conclude that $W^\varphi_2(A)$ is indeed a flat lifting of $A$.

\end{proof}

\subsubsection{Glueing to a twisted Witt scheme}

In order to globalise the above construction we shall show that it behaves well with respect to localisation.  For this purpose, we prove the following lemma resembling the computation given in \cite[Lemma 5.1]{bloch_ktheory}.

\begin{lemma}\label{lemma:opens}
Let $A$ be a $k$-algebra.  For any lifting $(f,c) \in W^\varphi_2(A)$ of $f \in A$ the natural homomorphism $i_f: W^\varphi_2(A)_{(f,c)} \to W^\varphi_2(A_f)$ is an isomorphism.
\end{lemma} 
Before proceeding to the proof, we present a series of useful equalities in $W^\varphi_2(A)$ which follows directly from addition and multiplication formulas: 
\begin{align}
	(f,c)^{-1} = (1/f,-c/f^2) \label{eq:inverse}\\
	(a,b)^n = (a^n,na^{n-1}b). \label{eq:power}
\end{align}
\begin{proof}[Proof of Lemma \ref{lemma:opens}]
The proof boils down to showing that $i_f$ is bijective.  To prove it, we first apply equalities (\ref{eq:inverse}) and (\ref{eq:power}) to obtain:
\[
	i_f\left(\frac{(a,b)}{(f,c)^n}\right) = (a,b) \mmul \left(\frac{1}{f},\frac{-c}{f^2}\right)^n = (a,b) \mmul \left(\frac{1}{f^n},\frac{-nc}{f^{n+1}}\right) = \left(\frac{a}{f^n},\frac{fb-nac}{f^{n+1}}\right),
\]
and then carry on with a straightforward inspection using the following two facts: (i) $\frac{a}{f^n} = 0$ if and only if $f^s a = 0$ for some $s \in \mathbb{N}$, (ii) the mapping $A_f \ni d \mapsto d - \frac{nac}{f^{n+1}}$ is bijective for any choice of parameters $(n,a,c)$.  
\end{proof}

We are now ready to prove the theorem.

\begin{thm}\label{thm:split_w2k} There exists a functor $W^{\varphi}_2$ from the category $\mathrm{Sch}^{split}_k$ to the category $\mathrm{Sch}_{W_2(k)}$ of schemes over the ring of second Witt vectors such that for every Frobenius split variety $(X,\varphi_X)$ the object $W^{\varphi}_2((X,\varphi_X))$ is a flat lifting of $X$ to $W_2(k)$. 
\end{thm}
\begin{proof}
In case of affine schemes $(\Spec(A),\phi)$ it is sufficient to use the construction of $W_2^\phi(A)$.  In general, we proceed as follows.  For a scheme $(X,\cc{O}_X)$ we consider a ringed space $(X,W^\varphi_2(\cc{O}_X))$ where $W^\varphi_2(\cc{O}_X)$ is a sheaf equal to $\cc{O}_X \times \cc{O}_X$ set theoretically and with ring structure defined by the assignment $U \mapsto W^\varphi_2(\cc{O}_X(U))$.  The sheaf transition maps are induced by the functoriality of the the construction given in \cref{thm:construction_split}.  Consequently, \cref{lemma:opens} proves that locally over the affine subset $V = \Spec(A)$ of $(X,\cc{O}_X)$, the sheaf $W^\varphi_2(\cc{O}_X)_{|V}$ is isomorphic to $\wt{W^\varphi_2(\cc{O}_X(V))}$ and therefore $(X,W^\varphi_2(\cc{O}_X))$ is in fact a scheme over $W_2(k)$.    
\end{proof}

\subsection{Alternative view on the construction}

As suggested by Piotr Achinger, the above construction might be alternatively described as a certain subscheme of the second Witt scheme $(X,W_2(\cO_X))$.  Firstly, observe that $W_2(\cO_X)$ is an algebra extension of $\cO_X$ by the ideal $F_{X*}\cO_X$.  The splitting $\varphi$ gives rise to a subideal of $F_{X*}\cO_X$ given by $\Ker(\varphi) \isom B^1_X$.  We define $W^{\phi,alt}_2(\cO_X)$ to be the quotient sheaf $W_2(\cO_X)/\Ker(\varphi)$.  We now prove:

\begin{prop}
The constructions $W^\varphi_2(\cO_X)$ and $W^{\varphi,alt}_2(\cO_X)$ give rise to isomorphic schemes. 
\end{prop}
\begin{proof}
Indeed, the construction of $W^{\varphi,alt}_2(\cO_X)$ can be described by the following diagram:
\begin{displaymath}
    \xymatrix{
    	& B^1_X \ar[r]^\isom \ar[d] & \Ker(\phi) \ar[d] & & & \\
        0 \ar[r] & F_{X*}\cc{O}_X \ar[r]\ar[d]^{\varphi} & W_2(\cc{O}_X) \ar[d]^{c_\varphi}\ar[r] & \cO_X \ar[r]\ar[d] & 0 \\
        0 \ar[r] & \cO_X \ar[r] & W^{\varphi,alt}_2(\cc{O}_X) \ar[r] & \cO_X \ar[r] & 0, }
\end{displaymath}
where $c_\varphi : W_2(\cO_X) \to W^{\varphi,alt}_2(\cO_X)$ is a ring homomorphism. Therefore, we see that the operations in $W^{\varphi,alt}_2(\cO_X)$ are given by the lifting to $W_2(\cO_X)$ (e.g., $\id \times F$ given in \cref{def:twisted_witt}), followed by the corresponding operation in $W_2(\cO_X)$ and the reduction $c_\varphi$.  This is exactly the description given by the diagrams in \cref{def:twisted_witt}.
\end{proof}


\section{Frobenius liftability} \label{sec:frobenius_liftability}

We now proceed to the investigation of Frobenius liftability of affine schemes.  We begin with functoriality of obstruction classes for lifting Frobenius, and then give a few examples.  Consequently we present a computational criterion for Frobenius liftability of affine complete intersections which we then apply to the case of ordinary double points in dimension $\geq 4$ and canonical surface singularities.

\subsection{Functoriality of obstructions to lifting Frobenius}

As a direct corollary of \cref{lem:obstruction_morphisms_func} we obtain the following result:

\begin{lemma}\label{lem:func_frobenius}
For any scheme $X/k$ together with a $W_2(k)$-lifting $\wt{X}$ there exists an obstruction $\sigma^F_X \in \Ext^1(F_{X/k}^*\Cot_{X^{(1)}/k},\cO_X)$ to lifting of relative Frobenius $F_{X/k} : X \to X^{(1)}$ to an infinitesimal flat thickening $\wt{F}_{X/k} : \wt{X} \to \wt{X}^{(1)}$ which satisfies the following functoriality property.  For any $\wt{g} : \wt{X} \to \wt{Y}$ restricting to a given $g: X \to Y$ the obstructions $\sigma^F_X$ and $\sigma^F_Y$ satisfy:
\begin{displaymath}
    \xymatrix{
        Lg^*F_{Y/k}^*\Cot_{Y^{(1)}/k} \ar[r]^{Lg^*\sigma^F_Y}\ar[d]^{F_{X/k}^*dg^{(1)}} & Lg^*\cO_Y[1] \ar[d]^{\isom} 	& & F_{Y/k}^*\Cot_{Y^{(1)}/k} \ar[r]^{\sigma^F_Y}\ar[d]^{d_g^{(1)}} & \cc{O}_Y[1] \ar[d]^{g^{\#}[1]} \\
        F_{X/k}^*\Cot_{X^{(1)}/k} \ar[r]^-{\sigma^F_X} & \cc{O}_X[1] 				& & Rg_*F_{X/k}^*\Cot_{X^{(1)}/k} \ar[r]^-{Rg_*\sigma^F_X} & Rg_*\cc{O}_X[1],}
\end{displaymath}
where $d_g^{(1)} : F_{Y/k}^*\Cot_{Y^{(1)}/k} \to Rg_*F_{X/k}^*\Cot_{X^{(1)}/k}$ is the adjoint of the mapping $F_{X/k}^*dg^{(1)}$. 
In particular, if $g$ is affine and $g^\#: \cO_Y \to g_*\cO_X$ splits then the Frobenius lifting of $X$ descends to $Y$.
\end{lemma}
\begin{proof}
We apply \cref{lem:obstruction_morphisms_func} to the case $X = Y$, $Y = Y^{(1)}$, $X' = X$, $Y' = X^{(1)}$, $f = F_{Y/k}$, $f' = F_{X/k}$, and then use the adjunction as in \cref{lem:obstruction_schemes_func}.  The second part follows from the functoriality diagram and the existence of a splitting of $g^\#$.
\end{proof}

As an exemplary application we obtain a corollary:

\begin{cor}\label{cor:frobenius_opens}
Let $j : U \to X$ be the inclusion of an open subset such that complement $Z = X \setminus U$ is of codimension $\geq 3$ and $X$ satisfies property $S_3$ at any point of $Z$.  Then, $X$ admits a Frobenius lifting if and only if $U$ does.
\end{cor}
\begin{proof}
Clearly, it suffices to show that Frobenius liftability of $U$ implies liftability for $X$.  For this purpose we observe that by \cref{lem:deformation_opens} the deformation functors of $U$ and $X$ are isomorphic.  Hence, any lifting $\wt{U} \in \Def_U(W_2(k))$ extends to a lifting $\wt{X} \in \Def_X(W_2(k))$ and consequently we may apply functoriality of obstructions \cref{lem:func_frobenius} to the inclusion $\wt{U} \to \wt{X}$.  We obtain a diagram:
\begin{displaymath}
    \xymatrix{
	F_{X/k}^*\Cot_{X^{(1)}/k} \ar[r]^{\sigma^F_X}\ar[d]^{d_j^{(1)}} & \cc{O}_X[1] \ar[d]^{j^{\#}[1]} \\
	Rj_*F_{U/k}^*\Cot_{U^{(1)}/k} \ar[r]^-{Rj_*\sigma^F_U} & Rj_*\cc{O}_U[1] .}
\end{displaymath}
By assumption $\sigma^F_U = 0$ and therefore $j^{\#} \comp \sigma^F_X = 0$.  To conclude that $\sigma^F_X = 0$, it suffices to show a natural mapping:
\[
\Ext^2(F_{X/k}^*\Cot_{X^{(1)}/k},\cO_X) \ra \Ext^2(F_{X/k}^*\Cot_{X^{(1)}/k},Rj_*\cO_U)
\]  
coming from the long exact sequence of $\Ext^\bullet(F_{X/k}^*\Cot_{X^{(1)}/k},-)$ groups associated to the distinguished triangle 
\[
\cO_X \ra Rj_*\cO_U \ra K_j \ra \cO_X[1]
\]
is injective.  This follows from the vanishing $\Ext^1(F_{X/k}^*\Cot_{X^{(1)}/k},K_j)=0$, which we prove along the lines of the proof of \cref{lem:deformation_opens} (this requires that any point in the complement $X \setminus U$ satisfies the property $S_2$).
\end{proof}

The simplest example of a Frobenius liftable scheme is given by the following.

\begin{example}[Frobenius liftability of toric varieties]\label{lem:toric}
Every affine toric variety $\Spec(A)/k$ is Frobenius liftable.
\end{example}
\begin{proof}
It is well-known that $A$ is isomorphic to a $k$-algebra $k[M]$ associated to a finitely generated monoid $M$. Its flat lifting is given by $W_2(k)[M]$ and the corresponding Frobenius lift is induced by the monoid mapping $M \ni m \mapsto pm$. 
\end{proof}


\subsection{Computational criterion}

In this section we present an explicit condition for the existence of a lifting of the Frobenius morphism to a fixed $W_2(k)$-lifting 
\[
\Spec(W_2(k)[x_1,\ldots,x_n]/(\wt{f_1},\ldots,\wt{f_m}))
\]
of an affine scheme $\Spec(k[x_1,\ldots,x_n]/(f_1,\ldots,f_m))$ (cf. \cref{lemma:def_ci}).  Then, we derive a computational criterion for the Frobenius liftability in the case of affine complete intersection schemes (in particular hypersurface singularities). 

Before stating the criteria we need some auxiliary definitions and notation.  We denote by $R$ the polynomial ring $k[x_1,\ldots,x_n]$, by $\wt{R}$ its natural $W_2(k)$-lifting $W_2(k)[x_1,\ldots,x_n]$ (unique up to isomorphism) and by $\wt{F} : \wt{R} \to \wt{R}$ the lifting of Frobenius morphism given by the formula: 
\[
\sum a_I x^I \mapsto \sum \sigma(a_I) x^{pI},
\]
where $I = (i_1,\ldots,i_n)$ is a multi-index, $x^I$ denotes the monomial $x_1^{i_1} \cdots x_n^{i_n}$ and $\sigma$ is the Witt vector Frobenius described in \cref{sec:notation}.  Moreover, let $w : \wt{R} \to \wt{R}$ be the base change homomorphism over $\sigma : W_2(k) \to W_2(k)$ defined by:
\[
w\left(\sum a_I x^I\right) = \sum \sigma(a_I) x^{I}.
\]

For any $W_2(k)$-flat ring $B = W_2(k)[x_1,\ldots,x_n]/(\wt{f_1},\ldots,\wt{f_m})$ we see by \cref{cor:lift_criterion} that any element belonging to $pB$ can be uniquely identified (by multiplication by $p$) with an element of $A = B/p$.  In particular for any element $\wt{f} \in \wt{R}$ there exists a unique element $P(\wt{f}) \in R$ satisfying $pP(\wt{f}) = \wt{F}(\wt{f}) - \wt{f}^p$.

\begin{example}
It turns out we have already seen an example of a polynomial $P(f)$.  Namely, in \cref{prop:non_liftability_quadric} we defined a polynomial:
\[
P_{f_{2n}} =  \sum_{\substack{{i_1+\ldots+i_n = p} \\ i_1,\ldots,i_n \neq p}} \frac{(p-1)!}{i_1! \cdots i_n!}(x_1x_2)^{i_1} \cdots (x_{2n-1}x_{2n})^{i_n}.
\]
This is exactly $-P(x_1x_2 + \ldots + x_{2n-1}x_{2n})$.
\end{example}

We are now ready to present the first computational condition for the existence of a Frobenius lift.

\begin{lemma}\label{lem:frob_for_given}
Suppose $B = W_2(k)[x_1,\ldots,x_n]/(\wt{f_1},\ldots,\wt{f_m})$ is a $W_2(k)$-flat lifting of the affine algebra $A = k[x_1,\ldots,x_n]/(f_1,\ldots,f_m)$.  Then, the Frobenius morphism of $A$ lifts if and only if there exists a sequence of elements $h_k \in R$ such that for every $i \in \{1,\ldots,m\}$ the element $P(\wt{f_i}) = \frac{\wt{F}(\wt{f}) - \wt{f}^p}{p} \in R$ satisfies:
\[
P(\wt{f_i}) = \sum_{1 \leq k \leq n} \left(\frac{\partial f_i}{\partial x_k}\right)^p h_k \pmod{f_1,\ldots,f_m}.
\]
\end{lemma}
\begin{proof}
Let $I \lhd k[x_1,\ldots,x_n]$ be the ideal generated by the polynomials $(f_1,\ldots,f_m)$ and $J \lhd W_2(k)[x_1,\ldots,x_n]$ be the ideal generated by the their lifts $(\wt{f_1},\ldots,\wt{f_m})$.  Every Frobenius lifting of $B = \wt{R}/J$ comes from a homomorphism $h : \wt{R} \to \wt{R}$ over $\sigma : W_2(k) \to W_2(k)$, given by the assignment $x_i \mapsto x_i^p + j_i + p\cdot h_i$ for $j_i \in J$, and $W_2(k) \ni a \mapsto \sigma(a)$.  In order to get an induced mapping $B \to B$ the ideal $J$ should be mapped to itself.  This boils down to the following computation:
\begin{align}
h(\wt{f_i}) & = w(\wt{f_i})(x_1^p + j_1 + p\cdot h_1,\ldots,x_n^p + j_n + p\cdot h_n) \nonumber\\
			& = w(\wt{f_i})(x_1^p,\ldots,x_n^p) + p \cdot \sum_{1 \leq k \leq n} \frac{\partial w(\wt{f_i})}{\partial x_k}(x_1^p,\ldots,x_n^p) \cdot h_k \label{step:taylor} \\
			& = \wt{F}(\wt{f_i}) + p \cdot \sum_{1 \leq k \leq n} \left(\frac{\partial f_i}{\partial x_k}\right)^p \cdot h_k \nonumber\\
			& = (\wt{F}(\wt{f_i}) - \wt{f_i}^p) + \wt{f_i}^p + p \cdot \sum_{1 \leq k \leq n} \left(\frac{\partial f_i}{\partial x_k}\right)^p \cdot h_k \nonumber\\
			& =  p \cdot \left(P(\wt{f_i}) + \sum_{1 \leq k \leq n} \left(\frac{\partial f_i}{\partial x_k}\right)^p \cdot h_k \right) \pmod{J} \label{step:up_to_pth_power},
\end{align}
where in (\ref{step:taylor}) we applied the Taylor expansion of a polynomial, together with the fact that $j_i \in J$ and $p^2 = 0$, and in (\ref{step:up_to_pth_power}) we used the fact that $(\wt{F}(\wt{f_i}) - \wt{f_i^p}) = p \cdot P(\wt{f_i})$ for a unique polynomial $P(\wt{f_i}) \in R$. By \cref{cor:lift_criterion} we therefore see that existence of a Frobenius lifting for a $W_2(k)$-lifting $B$ is equivalent to existence of $h_i \in R$ such that:
\[
P(\wt{f_i}) = \sum_{1 \leq k \leq n} \left(\frac{\partial f_i}{\partial x_k}\right)^p h_k \pmod{I},
\]
for any $i \in \{1,\ldots,m\}$.  This finishes the proof.
\end{proof}

We now restrict ourselves to the case of affine complete intersection schemes and derive a computationally feasible criterion to check whether they possess a $W_2(k)$-lifting compatible with Frobenius.  
For this purpose we firstly observe that:
\[
P(\wt{f} + pg) - P(\wt{f}) = \frac{(\wt{F}(\wt{f} + pg) - (\wt{f} + pg)^p)}{p} - \frac{(\wt{F}(\wt{f}) - f^p)}{p} = \frac{\wt{F}(pg)}{p} = g^p,
\]
which implies that the choice of a different lifting $\wt{f}$ of $f$ leads to the change of $P(\wt{f})$ by an arbitrary $p$-th power $g^p$.  This leads to a criterion:

\begin{thm}\label{lem:complete_intersections_criterion}
The affine complete intersection algebra $A = k[x_1,\ldots,x_n]/(f_1,\ldots,f_m)$ is Frobenius liftable if and only if there exists a sequence of elements $h_k \in R$ and a sequence of elements $g_i \in R$ such that:
\begin{align}
P_{f_i} + g_i^p = \sum_{1 \leq k \leq n} \left(\frac{\partial f_i}{\partial x_k}\right)^p h_k \pmod{f_1,\ldots,f_m} \label{eq:condition},
\end{align}
where $P_{f_i}$ is a polynomial defined up to a $p$-th power and computed by the formula $P_{f_i} = P(\wt{f_i})$ for any lifting $\wt{f_i}$ of $f_i$.
\end{thm}
\begin{proof}
By \cref{lemma:def_ci} we know that flat liftings of complete intersection affine schemes are given by the liftings of the the associated regular sequences.  Therefore, by \cref{lem:frob_for_given} we need to find liftings $\{\wt{f_i}\}$ if $\{f_i\}$ such that there exist $\{h_k\}$
\[
P(\wt{f_i}) = \sum_{1 \leq k \leq n} \left(\frac{\partial f_i}{\partial x_k}\right)^p h_k \pmod{f_1,\ldots,f_m} \label{eq:condition},
\]  
But, as we mentioned, the polynomials $P(\wt{f_i})$ for different liftings of $f_i$ differ by an arbitrary $p$-th power.  This finishes the proof.
\end{proof}

As a corollary we obtain the criterion for Frobenius liftability of hypersurface singularities:
\begin{cor}\label{cor:criterion}
Suppose $H_f = \Spec(k[x_1,\ldots,x_n]/(f))$ is a hypersurface in $\Affine^n_k$.   Then $H_f$ admits a lift to $W_2(k)$ compatible with Frobenius if and only if there exists an element $g$ such that 
\[
P_f + g^p \in \left(f,\left(\frac{\partial f}{\partial x_k}\right)^p,\ldots,\left(\frac{\partial f}{\partial x_n}\right)^p\right).
\]
\end{cor}

\begin{remark} 
The equality above might be treated as a relation in the Frobenius push-forward $F_*\left(k[x_1,\ldots,x_n]/(f)\right)$.  From this perspective, the polynomial $P_f \in F_*\left(k[x_1,\ldots,x_n]/(f)\right)$ (module structure by $p$-th powers) is supposed to belong to a submodule $k[x_1,\ldots,x_n]/(f) \cdot 1 + (\frac{\partial f}{\partial x_k},\ldots,\frac{\partial f}{\partial x_n})\cdot F_*(k[x_1,\ldots,x_n]/(f))$ (in fact, $g^p = g \cdot 1 \in F_*\left(k[x_1,\ldots,x_n]/(f)\right)$).
\end{remark}

It turns out that the second form of the above criterion can be (quite) efficiently verified using $\mathsf{Macaulay 2}$.  The code is given at \url{http://www.mimuw.edu.pl/~mez/Macaulay2}.

\begin{remark}
The existence of Frobenius lifting depends on the choice of the $W_2(k)$-lifting.  
\end{remark}

Indeed, Frobenius liftability of any lifting of $\Spec(k[x_1,\ldots,x_n]/(f))$ would mean that $P_f + g^p$ belongs to $(f,(\frac{\partial f}{\partial x_k})^p,\ldots,(\frac{\partial f}{\partial x_n})^p)$ for any $g$ and consequently $(f,(\frac{\partial f}{\partial x_k})^p,\ldots,(\frac{\partial f}{\partial x_n})^p)$ contains all $p$-th powers.  This is already not true for $f = x_1^3 + x_2^3 + x_3^3$ in characteristic $p \equiv 1 \pmod{3}$.  In fact, a simple computation shows that
\[
x_1^p \not\in (x_1^3 + x_2^3 + x_3^3,x_1^{2p},x_2^{2p},x_3^{2p}).
\]


\subsection{Non-liftability of ordinary double points in any dimension}

This section is devoted to answering the question posed in \cite{bhatt_torsion}[Remark 3.14], whether the ordinary double points, in particular cones over smooth projective quadrics, admit a lift to $W_2(k)$ compatible with Frobenius in arbitrary dimensions. The negative answer is given by following \cref{thm:cones_bhatt}.  

\begin{thm}[Frobenius non-liftability of ordinary double points]\label{thm:cones_bhatt}
Let $n \geqslant 5$ be an integer and assume $p \geqslant 3$. Then the ordinary double points defined by the equation $f = x_1^2 + \ldots + x_n^2 + Q$ for $Q \in k[x_1,\ldots,x_n]_{\geq 3}$ does not admit a $W_2(k)$-lifting compatible with Frobenius.
\end{thm}
\begin{proof}
For the proof, we need the lemma given in \cite{Langer_2008} combined with the \cref{prop:non_liftability_quadric}.

\begin{lemma}[\cite{Langer_2008}, Proposition 3.1]\label{lem:langer}
Let $k$ be a field of characteristic $p \geqslant 3$ and $0 \leq e<p$ be an integer.  Then for any $d \leq N \cdot \frac{p-1}{2} - e$ we have
\[
\left((x_1^p,\ldots,x_N^p) : (\sum_{i = 1}^{N} x_i^2)^e\right)_d = \left(x_1^p,\ldots,x_N^p,(\sum_{i = 1}^{N} x_i^2)^{p-e}\right)_d
\]
in $k[x_1,\ldots,x_N]$.
\end{lemma}

By \cref{cor:criterion} we need to show that 
\[
P_f + g^p \not\in \left(f,\left(\frac{\partial f}{\partial x_k}\right)^p,\ldots,\left(\frac{\partial f}{\partial x_n}\right)^p\right)
\] 
for any $g \in k[x_1,\ldots,x_n]$.  Clearly, $\left(f,\left(\frac{\partial f}{\partial x_k}\right)^p,\ldots,\left(\frac{\partial f}{\partial x_n}\right)^p\right) \subseteq \left(f,x_1^p,\ldots,x_n^p\right)$ and therefore it suffices to prove that $P_f \not\in (f,x_1^p,\ldots,x_n^p)$.  Assume the contrary, i.e., there exists an $h \in k[x_1,\ldots,x_n]$ such that:
\[
P_f = P_{x_1^2 + \ldots + x_n^2} + P' = h \cdot (x_1^2 + \ldots + x_n^2 + Q) \pmod{x_1^p,\ldots,x_n^p},
\]
where $P'$ is a polynomial in $k[x_1,\ldots,x_n]_{>2p}$.  The ring $k[x_1,\ldots,x_n]/(x_1^p,\ldots,x_n^p)$ admits a grading (induced from the polynomial ring).  We may assume that the lowest grade $m$ of the polynomial $h$ satisfies $m \geq 2p-2$.  Indeed, if $m < 2p-2$ then by lowest grade comparison and \cref{lem:langer} for $N = 5$, $e = 1$ and $d = m$ we see that 
\[
h_{[m]} \in \left((x_1^p,\ldots,x_n^p) : (\sum_{i = 0}^{n} x_i^2)\right)_{m} = \left(x_1^p,\ldots,x_N^p,(\sum_{i = 1}^{N} x_i^2)^{p-1}\right)_m \subset (x_1^p,\ldots,x_n^p)_m 
\]
and therefore $h - h_{[m]}$ is also an appropriate choice of $h$ with the lowest grade at least $m+1$.  Consequently, by the $2p$ grade comparison we see that $P_{x_1^2 + \ldots + x_n^2}$ belongs to the ideal $(x_1^2 + \ldots + x_n^2,x_1^p,\ldots,x_n^p)$.  This gives a contradiction with the proof of \cref{prop:non_liftability_quadric}. Namely $P_{f_{n}}$ from \cref{prop:non_liftability_quadric} is the representation of the polynomial $P_{x_1^2 + \ldots + x_n^2}$ after standard linear transformation.
\end{proof}

This in fact implies that the methods resulting from \cite{bhatt_torsion} cannot be applied for ordinary double points.  However, Bhatt proved that crystalline cohomology over $k$ of homogeneous ordinary double points is not finitely generated by a different calculation.


\subsection{Canonical singularities in positive characteristic are Frobenius liftable}\label{sec:canonical_singularities}

In this section, we present two different approaches to proving that canonical singularities of surfaces are $F$-liftable.  The first one is a direct computation using the criterion given in \cref{cor:criterion}.  The second gives a slightly weaker result in the case of tame quotient singularities and exploits the structure of the quotient together with functoriality of obstructions for lifting Frobenius morphism.  We note that in any characteristic the canonical surface singularities over algebraically closed field are classified by the dual graphs of the minimal resolution, which in turn correspond to Dynkin diagrams of type A, D and E.  Therefore, we shall use the terms canonical surface singularities and ADE singularities interchangeably.

\subsubsection{Direct computational approach}

Firstly, we approach the $F$-liftability of ADE singularities by a direct computation.  By \cite{artin_canonical_charp} under the assumption $p \geq 7$ the ADE singularities are locally analytically equivalent to the germs of the following form:
  
\begin{table}[ht]
\centering
\caption{Models of canonical singularities in char. $p \geq 7$}
\label{fig:models_canonical}
\begin{tabular}{lccll}
\cline{1-2}
 Type & \multicolumn{1}{l}{Equation} & \\ \cline{1-2}
smooth & - & \\
$A_{n-1}$ 	& $x^n + y^2 + z^2 = 0$			& \\
$D_{n+2}$ 	& $x^2 + y^2z + z^{n+1} = 0$		& \\
$E_6$     		& $x^2 + y^3 + z^4 = 0$			& \\
$E_7$     		& $x^2 + y^3 + yz^3 = 0$			& \\
$E_8$     		& $x^2 + y^3 + z^5 = 0$			& \\ \cline{1-2}
\end{tabular}
\end{table}

Equipped with this classification we can approach the proof of the following:

\begin{thm}\label{thm:canonical_surface_frobenius}
For any algebraically closed field $k$ of characteristic $p \geq 7$ any affine scheme $X$ with canonical surface singularities is Frobenius liftable.
\end{thm}
\begin{proof}
Firstly, we observe that by \cref{lem:affine_deformation_Zariski_local} it suffices to prove that for any $x \in X$ the germ $\Spec(\cO_{X,x})$ is Frobenius liftable.  By Artin's approximation theorem and the analytical classification of canonical singularities there exists a diagram of schemes:
\begin{displaymath}
    \xymatrix{
    & (U,u) \ar[dl]_{\text{\'{e}tale}} \ar[dr]^{\text{\'{e}tale}} & \\
    \Spec(\cO_{X,x}) & & \Spec(\cO_{S_{\mathrm{ADE}},s}) ,}
\end{displaymath}
where $\cO_{S_{\mathrm{ADE}},s}$ is the local ring at $0$ of one of the affine models given in \cref{fig:models_canonical}.  Therefore, by \cref{lem:affine_deformation_etale_local} and \cref{lem:affine_deformation_Zariski_local}, to prove Frobenius liftability of $\Spec(\cO_{X,x})$ it suffices to show that the affine models given in \cref{fig:models_canonical} are $F$-liftable.  

The case of $A_{n-1}$ singularities simply follows from \cref{lem:toric} as $A_{n-1}$ singularities are toric.  Consequently, we treat the case of $D_{n+2}$ singularities.  The affine model of $D_{n+2}$ is given by the equation $x^2 + y^2z + z^{n+1} = 0$ and therefore by \cref{cor:criterion} we are left to show that:
\[
\sum_{\substack{{i+j+k = p} \\ i,j,k \neq p}} \frac{{p \choose i,j,k}}{p} x^{2i}(y^{2}z)^jz^{(n+1)k} \in (x^2 + y^2z + z^{n+1},x^p,y^pz^p,y^{2p} + (n+1)^pz^{np}).
\]
By substituting $u = x^2 = -y^2z - z^{n+1}$ we see that it suffices to prove that:
\[
\sum_{\substack{{i+j+k = p} \\ i,j,k \neq p}} \frac{{p \choose i,j,k}}{p} (-1)^i (y^2z + z^{n+1})^i(y^{2}z)^jz^{(n+1)k} \in ((y^2z + z^{n+1})^{\ceil{p/2}},y^pz^p,y^{2p} + (n+1)^pz^{np}).
\]

The left hand side of the above relation is a sum of elements of the form $(y^2z)^a \cdot z^{(n+1)b}$ for $a+b = p$ so it is enough to show that:
\[
(y^2z)^a z^{(n+1)b} \in ((y^2z + z^{n+1})^{\ceil{p/2}},y^pz^p,y^{2p} + (n+1)^pz^{np}),
\]
for any $a+b = p$.  The claim is clear for every $a \geq \ceil{p/2}$ and therefore we reason by contradiction.  Suppose $a \leq \floor{p/2}$ is maximal with respect to the relation
\[
(y^2z)^{a} z^{(n+1)(p-a)} \not\in ((y^2z + z^{n+1})^{\ceil{p/2}},y^pz^p,y^{2p} + (n+1)^pz^{np}).
\]
As $p-a \geq \ceil{p/2}$ we can write
\begin{align*}
(y^2z)^{a} \cdot z^{(n+1)(p-a)} & = (y^2z)^a z^{(n+1)\ceil{p/2}}z^{(n+1)(p-a-\ceil{p/2})} \\
& = (y^2z)^a z^{(n+1)(p-a-\ceil{p/2})} \cdot \sum_{1 \leq i \leq p} {\ceil{p/2} \choose i} (y^2z)^i z^{(n+1)(\ceil{p/2} - i)} \\
& = \sum_{1 \leq i \leq p} {\ceil{p/2} \choose i} (y^2z)^{a+i} z^{(n+1)(p - a - i)},
\end{align*}
which gives a contradiction with the definition of $a$.

We now proceed to the case of $E_7$.  We shall prove that
\[
\sum_{\substack{{i+j+k = p} \\ i,j,k \neq p}} \frac{{p \choose i,j,k}}{p} x^{2i}y^{3j}(yz^3)^k \in (x^2 + y^3 + yz^3,x^p,3y^{2p} + z^{3p},y^pz^{2p}).
\]
By the trick as above, it suffices to prove that for any $a+b = p$ we have
\[
(yz^3)^a y^{3b} \in ((y^3 + yz^3)^{\ceil{p/2}},3y^{2p} + z^{3p},y^pz^{2p}),
\]
which follows by analogous reasoning by contradiction.

We finish with case of $E_{2m}$ singularities given by the affine equation $x^2 + y^3 + z^{m+1}$ for $m = 3,4$.  By \cref{cor:criterion} to prove the existence of $W_2(k)$-lifting compatible with Frobenius it suffices to prove that 
\[
\sum_{i+j+k=p} \frac{{p \choose i,j,k}}{p} x^{2i}y^{3j}z^{(m+1)k} \in (x^2 + y^3 + z^{m+1},x^p,y^{2p},z^{mp}).
\]
After substituting $u = x^2$, $v = y^3$, $w = z^{m+1}$ and $s = u+v+w$ it suffices to show that:
\[
\sum_{i+j+k=p} \frac{{p \choose i,j,k}}{p} u^{i}v^{j}(s-u-v)^{k} \in (s,u^{\ceil{p/2}},v^{\ceil{2p/3}},(s-u-v)^{\ceil{\frac{mp}{m+1}}}) = (s,u^{\ceil{p/2}},v^{\ceil{2p/3}},(u+v)^{\ceil{\frac{mp}{m+1}}}).
\]
For this purpose, we prove that any monomial $u^iv^j \in k[u,v]$ for $i+j = p$ appearing on the left hand side belongs to $(u^{\ceil{p/2}},u^{\ceil{2p/3}},(u+v)^{\ceil{\frac{mp}{m+1}}})$, which is equivalent to the surjectivity of the linear mapping defined by:
\[
k[u,v]_\floor{p/2} \times k[u,v]_\floor{p/3} \times k[u,v]_\floor{\frac{p}{m+1}} \ni (f,g,h) \mapsto f \cdot u^\ceil{p/2} + g \cdot u^\ceil{2p/3}+h \cdot (u+v)^\ceil{mp/(m+1)}.
\]
After writing this in the basis given by monomials $u^kv^j$ ordered lexicographically, to show surjectivity it is sufficient to show that the determinant of the matrix $\left[{\ceil{mp/(m+1)} \choose s+i-j}\right]_{1 \leq i,j \leq n}$, for $n = p - 1- \floor{p/2} - \floor{p/3}$ is non-zero.
By formula (2.1) in \cite{ck_det}, it is equal to:
\[
\det\left[{\ceil{\frac{mp}{m+1}} \choose s+i-j}\right]_{i,j \leq n} = \prod_{i = 1}^n \frac{(\ceil{\frac{mp}{m+1}}+i-1)!(i-1)!}{(s+i-1)!(\ceil{\frac{mp}{m+1}}-s+i-1)!},
\]
and therefore we are done by the inequality $\ceil{\frac{mp}{m+1}} + n - 1 < p$.
\end{proof}

\subsubsection{Quotient singularity approach for $F$-liftability of canonical singularities}\label{sec:quotient}

Here, we approach the $F$-liftability of canonical singularities of surfaces using a different method based on functoriality of obstructions.  We begin with a general result concerning $F$-liftability of quotients of smooth schemes by finite groups, for which the action can be lifted to characteristic $0$ (in fact, lifting mod $p^2$ would be sufficient).  Subsequently, we investigate to which extent our method can be applied in the case of canonical surfaces singularities.  Our analysis is based on the results of \cite[Section 4]{liedtke_satriano} describing whether ADE singularities admit a structure of a quotient singularity.

We precede our actual considerations by a few essential and standard remarks concerning base change for invariants of actions of finite groups.  The proofs are based on the existence of so-called Reynolds operator.

\begin{lemma}\label{lem:reynolds}
Let $G$ be a finite group and let $R$ be a commutative ring such that $|G|$ is invertible in $R$.  Then, any $R[G]$-module $M$ projective as an $R$-module is projective as an $R[G]$-module.
\end{lemma}
\begin{proof}
In order to prove that $M$ is projective we need show that for any diagram of $R[G]$-modules:
\begin{displaymath}
    \xymatrix{
          & M \ar[d]^{f} &  \\
         P \ar[r]^{p} & N \ar[r] & 0   }  
\end{displaymath}
there exists an $R[G]$-linear homomorphism $s : M \to P$ such that $f = p \comp s$.  By the assumptions we obtain an $R$-linear homomorphism $\sigma : M \to P$ such that $f = p \comp \alpha$.  Consequently, by a direct computation we obtain that a homomorphism defined by the formula (i.e., the Reynolds operator):
\[
s(m) = \frac{1}{|G|} \sum_{g \in G} g\sigma(g^{-1}m)
\]
is $R[G]$-linear and satisfies $f = p \comp s$.
\end{proof}

We now apply the above result to prove:

\begin{lemma}\label{lem:base_change}
Let $G$ be a finite group and let $R$ be a commutative ring such that $|G|$ is invertible in $R$.  Let $M$ be a projective $R$-module together with an action of $G$.  Then, for any homomorphism $R \to S$ the natural mapping $M^G \otimes_R S \to (M \otimes_R S)^G$ is an isomorphism.
\end{lemma}
\begin{proof}
Firstly, by \cref{lem:reynolds} we observe that $M$ is a projective $R[G]$-module.  Therefore, there exists an $R[G]$-module $N$ such that $M \oplus N \isom R[G]^{I}$ for some set of indices $I$.  Consequently, we obtain the following diagram:
\begin{displaymath}
	\xymatrix{
		(M \oplus N)^G \otimes_ R S \ar[d]^{\isom} \ar[r]^{\isom} &  M^G \otimes_R S \oplus N^G \otimes_R S \ar[d] \\
		(R[G]^I)^G \otimes_R S \ar[r] & (M \otimes_R S)^G \oplus (N \otimes_R S)^G \isom (R[G] \otimes_R S)^G,  
	}  
\end{displaymath}
which reduces our claim to the case of free modules.  We finish by the following simple sequence of identifications: 
\[
(R[G]^I)^G \otimes_R S \to R^I \otimes_R S \to S^I \to (S[G]^I)^G \to (R[G]^I \otimes_R S)^G.
\]
\end{proof}

As a corollary we obtain:

\begin{cor}\label{cor:quotient}
Let $G$ be a finite group and let $R$ be a ring such that $|G|$ is invertible in $R$.  Then, for every $G$-action over $\Spec(R)$ on $\Affine^n_{\Spec(R)}$ the quotient $\Affine^n_{\Spec(R)}/G$ is flat over $R$.  Moreover, for every ring homorphism $R \to S$ the natural map:
\[
\left(\Affine^n_{\Spec(R)}/G\right) \times_{\Spec(R)} \Spec(S) \to \Affine^n_{\Spec(S)}/G.
\]
is an isomorphism.  
\end{cor}
\begin{proof}
Using the same idea as in \cref{lem:reynolds} we prove that the ring of invariants $R[x_1,..,x_n]^G$ is a direct sum of a flat $R$-module $R[x_1,\ldots,x_n]$.  This implies that $\Affine^n_{\Spec(R)}/G = \Spec(R[x_1,..,x_n]^G)$ is flat over $R$.  The base change property is a direct consequence of \cref{lem:base_change}.
\end{proof}

We are now ready to approach the following result concerning $F$-liftability of quotient singularities.

\begin{lemma}[Frobenius liftability of mod $p$ reductions of quotients]\label{thm:frobenius_quotient}
Let $G$ be a finite group and let $T$ be a spectrum of finitely generated $\ZZ$-algebra \'{e}tale over $\Spec(\ZZ[1/|G|])$.  Suppose, $G$ acts on $\Affine^n_T$ relatively to $T$.  Then, for every perfect field $k$ and a $k$-point $t \in T(k)$ the scheme $\Affine^n_t/G \isom (\Affine^n_T/G)_t$ is $W_2(k)$-liftable compatibly with Frobenius.
\end{lemma}
\begin{proof}
Let $t \in T(k)$ be $k$-point of $T$.  By the formal lifting property of the \'{e}tale map $T \to \Spec(\ZZ)$, any point $t \in T(k)$ can be lifted to a point $\wt{t} \in T(W_2(k))$.  Consequently, by the flatness of $\Affine^n_T/G \to T$ (see \cref{cor:quotient}) we obtain a $W_2(k)$-lifting $(\Affine^n_T/G)_{\wt{t}}$ of $(\Affine^n_T/G)_t$ fitting into a commutative diagram with cartesian squares: 
\begin{displaymath}
\xymatrix{
	\Affine^n_t \ar[rr]\ar[dd]\ar[rd]^{\pi_t} & & \Affine^n_{\wt{t}} \ar[dd]\ar[rd]^{\pi_{\wt{t}}} & \\
	& (\Affine^n_T/G)_t \isom \Affine^n_t/G \ar[ld]\ar[rr] & & (\Affine^n_T/G)_{\wt{t}} \ar[ld] \\
	\Spec(k) \ar[rr] & & \Spec(W_2(k)). &   }
\end{displaymath} 

By \cref{cor:quotient} the scheme $(\Affine^n_T/G)_t$ is isomorphic to $\Affine^n_t/G$ and $\pi_t$ is a quotient map $\Affine^n_t \to \Affine^n_t/G$.  Hence by the assumptions on $T$ the degree $\deg(\pi_t)$ (equal to the length of a generic orbit) is coprime to the characteristic of $k$ and therefore there exists a splitting operator of $\pi_t^\#$, which consequently allows us to apply \cref{lem:func_frobenius} (for the upper square with mappings $\pi_t$ and $\pi_{W_2(k)}$) to obtain our claim.  More precisely, the Frobenius lifting of $\Affine^n_{W_2(k)}$ descends to $(\Affine^n_T/G)_{W_2(k)}$, which is a $W_2(k)$-lifting of $(\Affine^n_T/G)_{W_2(k)}$.
\end{proof}

In order to apply \cref{thm:frobenius_quotient} to the case of canonical singularities of surfaces we need some results from \cite{liedtke_satriano}.  First, let us recall the definition:

\begin{defin}\label{def:quotient}
 A scheme over a field $k$ has \emph{linearly reductive quotient singularities} (resp. \emph{tame quotient singularities}) if it is \'{e}tale equivalent to a quotient of a smooth $k$-scheme by a finite linearly reductive group scheme (resp. finite \'{e}tale group scheme of order prime to the characteristic of k).
\end{defin}

It is proven in \cite[Prop. 4.2]{liedtke_satriano} that except for a few characteristics all ADE singularities fit into the above definition.  Moreover, from the proof we can infer the following result concerning characteristic $0$ liftability of group actions leading to ADE singularities:

\begin{prop}\label{prop:group_schemes_ade}
For any singularity type $\tau \in \{ A_{n-1}, D_{n+2}, E_6, E_7, E_8\}$ there exists a finite flat generically \'{e}tale group scheme $\Lambda_\tau/\Spec(\ZZ)$ such that for every field $k$ the type $\tau$ singularity over $k$ arises as the quotient $\Affine^2_k/\Lambda_k$.  The group scheme $\Lambda$ is \'{e}tale precisely over an open subset of $\Spec(\ZZ)$ where the corresponding group scheme quotient is tame (in the sense of the above definition).
\end{prop}

\begin{example}
For example $A_{n-1}$ singularity over a field $k$ of arbitrary characteristic arises as the quotient of $\Affine^2_k/\Lambda_k$ by the natural action of the group scheme $\mu_{n,k} \isom \Lambda_k = \Lambda \times_{\Spec(\ZZ)} \Spec(k)$ for $\Lambda = \mu_{n,\ZZ}$.  In this case, the scheme $\Lambda/\ZZ$ is \'{e}tale over $\Spec(\ZZ[1/n])$. 
\end{example}

Finally, the details of the behaviour of canonical surface singularities with respect to characteristic of the field $k$ can be summarised by the following table from \cite[Section 4]{liedtke_satriano}.  

\begin{table}[ht]
\centering
\caption{Classification of canonical singularities of surfaces}
\label{my-label}
\begin{tabular}{lcccll}
\cline{1-4}
          & \multicolumn{1}{l}{Linearly reductive quotient singularity} & \multicolumn{1}{l}{Tame quotient singularity} &  & \\ \cline{1-4}
$A_{n-1}$ & every $p$                                                   & $p \not| n$                                   & &  \\
$D_{n+2}$ & $p \geq 3$                                                  & $p \geq 3,p \not| n$                          & &  \\
$E_6$     & $p \geq 5$                                                  & $p \geq 5$                                    & &  \\
$E_7$     & $p \geq 5$                                                  & $p \geq 5$                                    & &  \\
$E_8$     & $p \geq 7$                                                  & $p \geq 7$                                    & &  \\ \cline{1-4}
\end{tabular}
\end{table}

Note that the last column indicates the open subset where the respective group scheme $\Lambda$ is \'{e}tale.  We are now ready to give an alternative proof of $F$-liftability of canonical singularities of surfaces in the case they are tame quotient singularities.

\begin{thm}[Tame version of \cref{thm:canonical_surface_frobenius}]
Suppose $X$ is an affine surface over an algebraically closed field $k$ such that its singularities are tame quotient canonical surface singularities.  Then, $X$ is Frobenius liftable.
\end{thm}
\begin{proof}
Again by \cref{lem:affine_deformation_etale_local} and \cref{lem:affine_deformation_Zariski_local} it suffices to address the case of tame quotients.  By \cref{prop:group_schemes_ade} and the tameness assumption there exists an open subset $U = \Spec(\ZZ[1/N]) \subset \Spec(\ZZ)$ and an \'{e}tale group scheme $\Lambda/U$ of rank coprime to $N$ such that $X \isom \Affine^2_u/\Lambda_u$ for a certain point $u \in U(k)$.  Every finite \'{e}tale group scheme is \'{e}tale locally constant and therefore there exists an \'{e}tale morphism $T \to U$ such that $\Lambda_T \mydef \Lambda \times_U T$ is isomorphic to a constant group scheme associated to a finite group $G$.  As $k$ is algebraically closed, the point $u \in U(k)$ lifts to $t \in T(k)$.  By \cref{thm:frobenius_quotient} applied for $G$ acting on $\Affine^2_T$ the fibre $\Affine^2_u/\Lambda_u \isom \Affine^2_t/G \isom (\Affine^2_T/G)_t$ is $F$-liftable.  This finishes the proof. 
\end{proof}

\begin{example}
A $D_{n+2}$ singularity arises as the quotient of $\Affine^2_k$ by the action of the binary dihedral group scheme $\mathrm{BD}_n$ of rank $4n$ induced by the natural operation of the matrices:
\[
\begin{bmatrix}
    \xi_{2n} & 0 \\
    0 & \xi_{2n}^{-1}
\end{bmatrix} \textrm{ and }
\begin{bmatrix}
    0 & -1 \\
    1 & 0
\end{bmatrix},
\] 
where $\xi_{2n}$ denotes the $2n$-th primitive root of unity. This means that the model of $D_{n+2}$ is tame over $\Spec(\ZZ[1/2n])$ and the associated \'{e}tale group scheme is trivialized by the covering induced by $\ZZ \to \ZZ[1/2n][\xi_{2n}]$.  For more details, the reader is referred to \cite[Section 4]{liedtke_satriano}.

\end{example}


\section{$W_2(k)$-liftability and $F$-liftability compared to standard $F$-singularities}\label{sec:relation_f-singularities}

The following section contains the comparison of the classical $F$-singularity types, i.e., $F$-regularity, $F$-purity and $F$-rationality with the notions of $W_2(k)$-liftability and Frobenius liftability.  We work with $F$-finite rings, i.e, characteristic $p>0$ rings such that $F_*R$ is a finite $R$-module.  In particular, the $F$-finiteness assumption is satisfied in the case of essential finite type algebras over a field $k$ of characteristic $p$.  We begin by recalling a few necessary definitions and criteria.  Note that we freely identify affine schemes with their associated rings.
\begin{defin}[$F$-regularity]
Let $R$ be a reduced $F$-finite ring of characteristic $p > 0$.  
\begin{enumerate}[i)]
\item We say that $R$ is $F$-regular if for any $c \in R$ the mapping $\varphi^e_c : R \to F^e_*R$ defined by the formula $1 \mapsto c$ splits for some $e \gg 0$.
\item We say that $R$ is $F$-pure if the mapping $F : R \to F_*R$ splits, i.e., the scheme $\Spec(R)$ is Frobenius split.
\end{enumerate}
\end{defin}

It turns out that both of above properties can be verified locally, i.e., $R$ is $F$-pure (respectively strongly $F$-regular) if and only if for any maximal $\ideal{m} \in \Spec(R)$ the local ring $R_{\ideal{m}}$ is $F$-pure (respectively strongly $F$-regular).  In the case of quotients of $F$-finite regular rings the above properties can be efficiently verified using the following criteria.  

\begin{lemma}[Fedder's and Glassbrenner's criteria]\label{lem:fedder}
Let $S$ be a F-finite regular ring such that $F_*S$ is a free $S$-module and let $R = S/I$.  Then $R$ is $F$-split at $\ideal{m} \supset I$ if and only if $I^{[p]} : I \not\subset \mathfrak{m}^{[p]}$.  Moreover, let $s$ be an element of $S$ not in any minimal prime of $I$ such that $R_s$ is regular.  Then, $R$ is strongly $F$-regular at $\ideal{m}$ if and only if there exists an $e \in \NN$ such that $s(I^{[p^e]} : I) \not\subset \mathfrak{m}^{[p^e]}$.
\end{lemma}
\begin{proof}
For the proof see \cite[Theorem 2.3]{donna_glassbrenner}.
\end{proof}

Note that the assumptions of the above lemma are satisfied for $S = k[x_1,\ldots,x_n]$ and $S$ regular local.  Moreover, in the special case $I = (f)$ the colon ideal $I^{[p^e]} : I = (f^{p^e-1})$ and therefore $F$-purity boils down to an efficiently verifiable criterion $f^{p^e-1} \in \ideal{m}^{[p^e]}$.

\begin{defin}[$F$-rational]
Let $(R, \ideal{m})$ be a $d$-dimensional local ring of characteristic $p > 0$.
\begin{enumerate}[i)]
\item We say that $R$ is $F$-rational if $R$ is Cohen-Macaulay and if for any $c \in R^\circ$, there exists $e \in \mathbb{N}$ such that $c F^e : H^d_{\ideal{m}}(R) \to H^d_{\ideal{m}}(R)$ is injective.
\end{enumerate}
\end{defin}

We shall need the following lemma.

\begin{lemma}[Watanabe $a$-invariant]\label{lem:watanabe}
Let $(R,\ideal{m})$ be a Cohen-Macaulay $\NN$-graded ring of dimension $d$ such that the punctured spectrum $\Spec(R) \setminus \{\ideal{m}\}$ is $F$-rational.  Then, if the invariant $a(R) \mydef \max \{ i \in \ZZ : [H^d_{\ideal{m}}(R)]_i \neq 0 \}$ satisfies the inequality $a(R) < 0$ then for any $n \gg 0$ the Veronese subring $R^{(n)}$ is $F$-rational.
\end{lemma}
\begin{proof}
See \cite[Proposition 5.1.1]{anuragsingh}.
\end{proof}

We are now ready to describe our results. They can be summarized by the following diagram indicating possible implications and referring to counterexamples in case they do not hold:

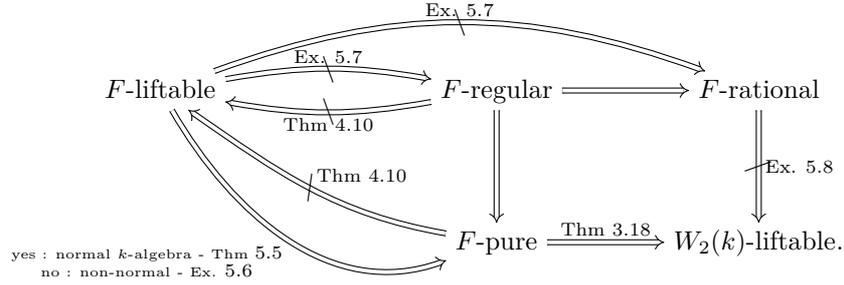
\begin{center}
\begin{tikzcd}[column sep=large,arrows=Rightarrow]
\text{$F$-liftable} \arrow[degil,out=20,in=160]{rrr}{\text{Ex. } \ref{ex:f-liftable_not_f-regular}} \arrow[out=10,in=170,degil]{rr}{\text{Ex. } \ref{ex:f-liftable_not_f-regular}} \arrow[out=-60,in=-160]{rrdd}[swap]{\substack{\text{yes : normal $k$-algebra - Thm } \ref{lem:flift_normal} \\ \text{no : non-normal - Ex. } \ref{ex:lift_not_imply_split}}} & & \text{$F$-regular} \arrow{dd}{} \arrow[out=190,in=-10,degil]{ll}{\text{Thm } \ref{thm:cones_bhatt}} \arrow{r}{} & \text{$F$-rational} \arrow[degil]{dd}{\text{Ex. } \ref{ex:frational_not_imply_w2k}
} \\
& & & \\ 
& & \text{$F$-pure} \arrow[out=170,in=-35,degil]{lluu}[swap]{\text{Thm } \ref{thm:cones_bhatt}} \arrow{r}{\text{Thm } \ref{thm:split_w2k}} &\text{$W_2(k)$-liftable}. \\
\end{tikzcd}
\end{center}

Firstly, we prove that an $F$-liftable normal scheme over $k$ is Frobenius split.

\begin{thm}\label{lem:flift_normal}
Let $X/k$ be normal scheme locally of finite type over a perfect field $k$, such that there exists a lifting $\wt{X}/W_2(k)$ together with a lifting $\wt{F}$ of Frobenius morphism $F$.  Then, $X$ is Frobenius split.  
\end{thm}
\begin{proof}
The following is a simple extension of a proof given in \cite{mehta_srinivas} covering the smooth case.  Let $n$ be the dimension of $X$ and let $U$ be the smooth locus of $X$. In case of smooth schemes liftability of Frobenius is equivalent to the existence of a splitting $\xi : \Omega^1_U \to Z^1_U$ of the Cartier mapping $C$ (see \cite{mehta_srinivas}).  After taking the $n$-th exterior power of $\xi$, we obtain a mapping $\wedge^n \xi : \Omega^n_U \to Z^n_U \isom F_*\Omega^n_U$ which splits the sequence:
\begin{displaymath}
    \xymatrix{
        0 \ar[r] & B^n_U \ar[r] & Z^n_U \isom F_*\Omega^n_U \ar[r] & \Omega^n_U \ar[r] & 0,}
\end{displaymath}
dual to the sequence 
\begin{displaymath}
    \xymatrix{
        0 \ar[r] & \cc{O}_U \ar[r] & F_*\cc{O}_U \ar[r] & B^1_U \ar[r] & 0,}
\end{displaymath} by applying $\cc{H}om_{\cc{O}_U}(-,\Omega^n_U)$.  Therefore, $U$ is Frobenius split.  Using normality of $X$, in fact $S_2$ property, one may extend this to a splitting of $X$.
\end{proof}

We shall now present an example which violates the hypothesis of \cref{lem:flift_normal} in case of non-normal schemes. 

\begin{example}\label{ex:lift_not_imply_split}
The scheme $\{x_1x_2(x_1+x_2) = 0\} \subset \Affine^2_k$ is Frobenius liftable but not $F$-pure.
\end{example}
\begin{proof}
We apply Fedder's criterion given in \cref{lem:fedder} and a simple fact that: 
\[
\left(x_1x_2(x_1+x_2)\right)^{p-1} \in (x_1^p,x_2^p).
\]
Frobenius liftability follows by observing that the canonical lifting of Frobenius of $W_2[x_1,x_2]$ given by $x_i \mapsto x_i^p$ preserves the ideal $x_1x_2(x_1 + x_2)$.
\end{proof}

\begin{example}\label{ex:f-liftable_not_f-regular}
The cone over an ordinary elliptic curve is $F$-liftable but not strongly $F$-regular or $F$-rational.
\end{example}
\begin{proof}
For for the sake of clarity, we focus on the case of a cone over Fermat cubic $C = \{x^3 + y^3 + z^3 = 0\} \subset \Affine^3_k$ in characteristic $p \equiv 1 \pmod{3}$.  Firstly, $F$-liftability follows from \cref{cor:criterion} by a simple direct computation.  

We verify that $C$ is strongly $F$-regularity by means of \cref{lem:fedder}.  Indeed, we observe that for any $s \in \ideal{m} = (x_1,x_2,x_3)$ we have $sf^{p^e-1} \in (x_1,x_2,x_3)^{[p^e]}$ and consequently the cone is not strongly $F$-regular at $\ideal{m}$.  Finally, the proof of the fact that the cone $C$ is not $F$-rational is the content of \cite[Example 2.6]{takagi_watanabe}.  
\end{proof}

\begin{example}\label{ex:frational_not_imply_w2k}
There exists an $F$-rational scheme which is not $W_2(k)$-liftable. 
\end{example}
\begin{proof}
Let $X$ be a scheme embedded into $\PP^n_k$ which is not $W_2(k)$-liftable.  By \cite[Theorem 2.3]{liedtke_satriano} we see that $Y \mydef \mathrm{Bl}_X(\PP^n_k)$ does not lift to $W_2(k)$, either.  We claim that a cone over sufficiently ample projective embedding of $Y$ satisfies our requirements, i.e., is $F$-rational and does not lift to $W_2(k)$. 

Indeed, using \cite{rulling_chatzistamatiou} and Leray spectral sequence we see that $H^i(Y,\cO_Y) = 0$ for $i>0$, and therefore we may apply \cref{cor:property_s_n} and \cref{prop:deformations_cones} to obtain a projective embedding $Y \subset \PP^N$ such that the cohomology groups $H^i(Y,\cO_Y(k))$ vanish for $i>0$ and $k > 0$, the cone $C_Y \mydef \Cone_{Y,\PP^N}$ is Cohen-Macaulay and does not lift to $W_2(k)$.  Consequently, in order to apply \cref{lem:watanabe} (which might require additional Veronese embedding) we are left to show that the positive degree part $[H^{n+1}_{\ideal{m}}(C_Y,\cO_{C_Y})]_{\geq 0}$ of the graded module $H^{n+1}_{\ideal{m}}(C_Y,\cO_{C_Y})$ is zero.  This follows from the identification $[H^{n+1}_{\ideal{m}}(C_Y,\cO_{C_Y})]_k = H^n(Y,\cO_Y(k)) = 0$ given in \cref{cor:cone_normality}.
\end{proof}


\section{$W_2(k)$ and Frobenius liftability of cones over projective schemes} \label{sec:cones}

Here, we present a few result concerning cones over projective schemes.  We begin with a series of classical results (see, e.g., \cite{schlessinger_rigidity,artin_lectures}).  Consequently, in \cref{sec:cones_frobenius} we extend the classical results with a comparison theorem for Frobenius liftability.

For the purpose of clarity we present a few definitions and results concerning cones over projective schemes and their deformation.  We begin by introducing the construction of a cone over a closed embedding of a scheme into the projective space $\PP^n \mydef \PP^n_k$.  For the sake of brevity we set $R = k[x_0,\ldots,x_n]$.  


\subsection{Construction of a cone over a projective scheme}

For any closed subscheme $X \subset \PP^n$ given by an ideal sheaf $\cc{I}$, we denote by $I_X$ the homogeneous ideal:
\[
I_X \mydef \bigoplus_{k \geq 0} H^0(\cO_{\PP^n},\cI(k)) \lhd R \isom \bigoplus_{k \geq 0} H^0(\PP^n,\cO_{\PP^n}(k)).
\]
From the long exact sequence of cohomology associated to the short exact sequence:
\[
0 \ra \bigoplus_{k \geq 0} \cI(k) \ra \bigoplus_{k \geq 0} \cO_{\PP^n}(k) \ra \bigoplus_{k \geq 0} \cO_X(k) \ra 0,
\] 
we see that $I_X$ arises as the first term in:
\[
0 \ra I_X \ra R = \bigoplus_{k \geq 0} H^0(\PP^n,\cO_{\PP^n}(k)) \xrightarrow{r_X} \bigoplus_{k \geq 0} H^0(X,\cO_X(k)) \ra \bigoplus_{k \geq 0} H^1(\PP^n,\cI(k)) \ra 0.
\]

We define the affine cone, as follows.

\begin{defin}\label{def:cone}
Suppose $X \subset \PP^n = \Proj(R)$ is a closed embedding of $k$-schemes.  We define \emph{the affine cone over $X \subset \PP^n$} to be an affine scheme:
\[
\Cone_{X,\PP^n} \mydef \Spec(\bigslant{R}{I_X}).
\]
\end{defin}
Directly from the definition we see that:
\begin{remark}
The ring $\bigslant{R}{I_X}$ admits a natural grading induced by the grading of $R$.
\end{remark}
For notational convenience, we denote by $j : U = \Cone_{X,\PP^n} \setminus \{\ideal{m}\} \to \Cone_{X,\PP^n}$ the inclusion of the complement of the vertex, and by $p : U \to X$ the natural $\mathbb{G}_m$-bundle given by
\[
U \isom \sSpec_X\left(\bigoplus_{k \in \bb{Z}} \cO_X(k)\right) \to X,
\] 
where the isomorphism follows from the Grauert's ampleness criterion (see EGA II 8.9.1).  We have the following standard proposition.

\begin{prop}\label{cor:cone_normality}
Suppose $X \subset \PP^n$ is a closed subscheme.  The affine scheme $C \mydef \Cone_{X,\PP^n}$ is normal if and only if $X$ is itself normal and the restriction mapping $r_X$ is surjective.  In this case, $C  \isom \Spec\left(\bigoplus_{k \in \ZZ} H^0(X,\cO_X(k))\right) \isom \Spec(H^0(U,\cO_U))$.  Moreover, for $i \geq 1$ the local cohomology around the vertex satisfy $H^{i+1}_\ideal{m}(C,\cO_C) = \bigoplus_{k \in \ZZ} H^i(X,\cO_X(k))$,
as graded $H^0(C,\cO_C)$-modules.
\end{prop}

\begin{remark}\label{remark:cone_veronese_normal}
Taking the cone over the $d$-th Veronese embedding of $X \subset \PP^n$ into $\PP^{{n \choose d}-1}$ amounts to considering the morphism of Veronese subrings 
\[
r_{X,d} : \bigoplus_{k \geq 0} H^0(\PP^n,\cc{O}_{\PP^n}(dk)) \to \bigoplus_{k \geq 0} H^0(X,\cO_{X}(dk)).
\]  
By the Serre's vanishing theorem we see that there exists $d$ such that for $k \geq d$ we have $H^1(X,I_X(k)) = 0$, and therefore there exists $d$ such that $r_{X,d}$ is surjective.  This means that for $d$ sufficiently large the cone over the Veronese embedding of $X$ is normal.
\end{remark}

Moreover, we have the corollary: 

\begin{cor}\label{cor:property_s_n}
Let $X$ be a smooth projective scheme of dimension $\dim(X) = n \geqslant 2$ satisfying $H^i(X,\cO_X) = 0$ for $i \in \{1,\ldots,n-1\}$.  Then, for a sufficiently ample embedding $X \subset \PP^N$ the cone satisfies property $S_{n+1}$. 
\end{cor}
\begin{proof}
This follows directly from the local cohomology part of \cref{cor:cone_normality} and Serre vanishing.
\end{proof}


\subsection{Deformation theory and cones}

We now proceed to the summary of deformation theory of cones.  Again, $X$ denotes a closed subscheme of $\PP^n$ and we assume that $C \mydef \Cone_{X,\PP^n}$ is normal at its vertex (cf. \cref{cor:cone_normality}).  For a thorough exposition (in fact equicharacteristic) expressed in classical terms of tangent sheaves one may take a look at \cite[Section 11 and 12]{artin_lectures}.

We begin with a proposition expressing the results of \cite[Theorem 12.1, p. 48; Lemma 12.1, p. 53]{artin_lectures}.

\begin{prop}\label{prop:hilb_def_cone_comparison}
There exists a morphism of deformation functors $\phi : \Hilb_{X,\PP^{n}} \xrightarrow{\phi} \Def_C$ defined by performing cone construction relatively. 
The tangent and obstruction mappings of $\phi$ satisfy the properties:
\begin{enumerate}[i)]
\item The tangent mapping $\Tan_\phi : \Tan_{\Hilb_{X,\PP^n}} \to \Tan_{\Def_C}$ can be identified with the canonical homomorphism
\[
H^0(X,\cN_{X/\PP^n}) \to \Coker\left(\bigoplus_{k \in \ZZ} H^0(X,\cT_{\PP^n|X}(k)) \to \bigoplus_{k \in \ZZ} H^0(X,\cN_{X/\PP^n}(k))\right),
\]\label{step1:hilb_def}
coming from the long exact cohomology sequence associated to
\[
0 \ra \bigoplus_{k \in \ZZ} \cT_X(k) \ra \bigoplus_{k \in \ZZ} \cT_{\PP^n|X}(k) \ra \bigoplus_{k \in \ZZ} \cN_{X/\PP^n}(k) \ra 0 \label{eq:long_exact_normal},
\]
\item The tangent mapping $\Tan_\phi$ is surjective if $\bigoplus_{k \neq 0} H^1(X,\cT_{X}(k)) = 0$. \label{step2:hilb_def}
\item The obstruction mapping $\Ob_\phi$ is injective. \label{step3:hilb_def}
\end{enumerate}
\end{prop}
\begin{proof}
For the proof, we refer to \cite{artin_lectures}.  The identification of the tangent space $\Tan_{\Def_C} = \Ext^1(\Cot_{C/k},\cO_C)$ with the given cokernel follows from an explicit calculation of the long exact sequence of $\Ext(-,\cO_C)$ groups for a distinguished triangle:
\[
Li^*\Cot_{\Affine^{n+1}/k} \ra \Cot_{C/k} \ra \Cot_{C/\Affine^{n+1}} \ra Li^*\Cot_{\Affine^{n+1}/k}[1]
\]
associated with the inclusion of schemes $i : C \to \Affine^{n+1}$.
\end{proof}

As a corollary we obtain the following result comparing the deformation theory of a projective scheme $X$ and the cone over its sufficiently large ample embedding.

\begin{cor}\label{cor:veronese_cone}
Let $X$ be a projective scheme.  Then, there exists a sufficiently large Veronese embedding of $X \subset \PP^n$ into $\PP^{N_d}$ such that the morphism of functor $\phi_d : \Hilb_{X,\PP^{N_d}} \to \Def_{\Cone_{X,\PP^{N_d}}}$ is smooth. 
\end{cor}
\begin{proof}
Firstly, by consideration following \cref{cor:cone_normality} we see that by taking sufficiently large Veronese embedding we may assume that the cone is in normal.  This allows us to apply \cref{prop:hilb_def_cone_comparison}.  By part \ref{step2:hilb_def}), we need to show that for sufficiently large $d$ we have $\bigoplus_{k \neq 0} H^1(X,\cT_{X}(kd)) = 0$.  This in turn follows from Serre's vanishing and Serre duality.
\end{proof}

By combining \cref{lem:comparing_hilb_def} with \cref{cor:veronese_cone} we obtain:
\begin{prop}\label{prop:deformations_cones}
For any smooth projective scheme $X$ satisfying $H^2(X,\cO_X) = 0$ and a sufficiently large $d$, the morphisms of deformation functors $\phi_d$ and $\psi_d$ given in a diagram:
\[
\Def_{\Cone_{X,\PP^{N_d}}} \xleftarrow{\phi_d} \Hilb_{X,\PP^{N_d}} \xrightarrow{\psi_d} \Def_X,
\]
are smooth.
\end{prop}

As a simple corollary, we see that a characteristic $p$ projective scheme $X$ admits a $W_2(k)$-lifting if and only if an appropriate cone does.

\subsubsection{Frobenius liftability of cones over sufficiently ample embeddings}\label{sec:cones_frobenius}

We shall now relate the Frobenius liftability of a normal projective scheme and a cone over its sufficiently ample embedding. 

\begin{prop}
Let $X$ be a normal projective scheme.  Then, for a sufficiently ample embedding $X \subset \PP^n$ the Frobenius liftability of $C = \Cone_{X,\PP^n}$ implies the Frobenius liftability of $X$.  Moreover, for any smooth scheme $X$ satisfying $H^1(X,\cO_X) = H^2(X,\cO_X) = 0$ Frobenius liftability of $X$ is equivalent to the Frobenius liftability of a cone over a sufficiently ample embedding.
\end{prop}
\begin{proof}
The first part we proceed as follows.  Let $C = \Cone_{X,\PP^n}$ be the cone over an embedding of $X$ such that the functor $\phi: \Hilb_{X,\PP^n} \to \Def_C$ is smooth (cf. \cref{cor:veronese_cone}).  Moreover, let $U$ be the cone without the vertex $\ideal{m}$.  Assume $C$ is Frobenius liftable, i.e, there exists a lifting $\wt{C} \in \Def_C(W_2(k))$ admitting a lifting of Frobenius.  By the smoothness of $\phi$ we obtain a (potentially non-unique) diagram:  
\[
\xymatrix{
C \ar[r] & \wt{C} & \\
	& U \ar[ul]\ar[r]\ar[d]^{p} & \wt{U} \ar[ul]\ar[d]^{\wt{p}} \\
	& X \ar[r] & \wt{X}.
}
\]
By \cref{lem:func_frobenius} we see that the obstruction classes to existence of a Frobenius lifting for $X$ and $U$ fit into a commutative diagram in $D(X)$:
\begin{displaymath}
    \xymatrix{
        F_{X/k}^*\Cot_{X^{(1)}/k} \ar[r]^{\sigma^F_X}\ar[d]^{d_p^{(1)}} & \cc{O}_X[1] \ar[d]_{p^{\#}[1]} \\
        Rp_*F_{U/k}^*\Cot_{U^{(1)}/k} \ar[r]^-{Rp_*\sigma^F_U} & Rp_*\cc{O}_U[1] \ar@/^-1.0pc/[u],}
\end{displaymath}
where the right most arrow is a splitting of $p^{\#}$ induced by the $\mathbb{G}_m$-bundle structure of $p$.  We assumed that $\wt{C}$ admits a lifting a therefore $\sigma^F_U = 0$.  By the existence of splitting $\sigma^F_X = 0$ ad therefore $X$ is Frobenius liftable. 

For the second, we assume that $X$ is Frobenius liftable, i.e., there exists a lifting $\wt{X} \in \Def_X(W_2(k))$ such that $\sigma^F_X = 0$.  By \cref{lem:comparing_hilb_def} and Serre vanishing we obtain an embedding of $X$ into $\PP^n$ such that $\psi : \Hilb_{X,\PP^n} \to \Def_X$ is smooth and $H^1(X,\cO_X(k)) = H^2(X,\cO_X(k)) = 0$ for any $k \neq 0$.  By smoothness of $\psi$ we see that $\wt{X}$ arises as an embedded deformation and therefore induces a compatible deformation $\wt{C} \in \Def_C(W_2(k))$ of the cone $C$.  By \cref{cor:frobenius_opens} in order to prove that $\wt{C}$ admits a Frobenius lift it suffices to infer it for $\wt{U}$.  By functoriality of obstructions to lifting Frobenius we obtain a diagram:
\begin{displaymath}
    \xymatrix{
        F_{U/k}^*Lp^{(1)*}\Cot_{X^{(1)}/k} \isom Lp^*F_{X/k}^*\Cot_{X^{(1)}/k} \ar[rr]^-(0.50){Lp^*\sigma^F_X}\ar[d]^{F_{U/k}^*dp^{(1)}} & & Lp^*\cO_X[1] \ar[d]^{\isom} \\
        F_{U/k}^*\Cot_{U^{(1)}/k} \ar[rr]^-{\sigma^F_U} & & \cc{O}_U[1].}
\end{displaymath}

We consider a Frobenius pullback of a distinguished triangle of cotangent complexes associated to the $\mathbb{G}_m$-bundle $p^{(1)} : U^{(1)} \to X^{(1)}$:
\[
F_{U/k}^*Lp^{(1)*}\Cot_{X^{(1)}/k} \ra F_{U/k}^*\Cot_{U^{(1)}/k} \ra F_{U/k}^*\Cot_{U^{(1)}/X^{(1)}} \isom \cO_U \ra F_{U/k}^*Lp^{(1)*}\Cot_{X^{(1)}/k}[1],
\]
where the isomorphism in the middle follows from the standard property of Zariski locally trivial $\mathbb{G}_m$-bundles.  By the long exact sequence of $\Ext(-,\cO_U)$ groups we obtain an exact sequence:
\[
\ldots \ra \Ext^1(\cO_U,\cO_U) \ra  \Ext(F_{U/k}^*\Cot_{U^{(1)}/k},\cO_U) \ra \Ext^1(F_{U/k}^*Lp^{(1)*}\Cot_{X^{(1)}/k},\cO_U) \ra \cdots,
\]
and therefore it suffices to prove that $\Ext^1(\cO_U,\cO_U) = 0$.  This follows from the isomorphisms $\Ext^1(\cO_U,\cO_U) \isom H^1(U,\cO_U) \isom H^1(X,\bigoplus_{k \in \ZZ} \cO_X(k))$ and the assumptions.

\end{proof}


\appendix

\section{Deformation theory}\label{appendix:deformation}

Oftentimes we resign from giving proofs and refer to the standard works, e.g, \cite{cotangent,schlessinger_functors,fantechi_manetti}.  In the part concerning cotangent complex and obstruction classes, the reader may freely ignore the technical proofs and focus on the statements concerning functoriality.


\subsection{Deformation theory in terms of cotangent complex}\label{appendix:deformation_cotangent}

We begin with the relevant topics in deformation theory (see \cite[Part I, Chapitre III]{cotangent}) based on the notion of cotangent complex $\Cot_{X/S}$.  We begin with the basic properties of Kodaira-Spencer class $K_{X/Y/S}$ and consequently we express the deformation obstruction classes in terms of $K_{X/Y/S}$ and other natural extensions.

\begin{defin}[Kodaira-Spencer class]
For any $f : X \to Y$ an $S$-morphism of schemes there exists a mapping $K_{X/Y/S} : \Cot_{X/Y} \to Lf^*\Cot_{Y/S}[1]$ in $D(X)$ called a \emph{Kodaira-Spencer class}.  It is defined as the connecting homomorphism in a natural distinguished triangle of cotangent complexes associated to $f$:
\begin{displaymath}
\xymatrix{
	Lf^*\Cot_{Y/S} \ar[r]^-{df} & \Cot_{X/S} \ar[r] & \Cot_{X/Y} \ar[rr]^-{K_{X/Y/S}} & & Lf^*\Cot_{Y/S}[1].}
\end{displaymath}
\end{defin}
The Kodaira-Spencer class satisfies the following functoriality property.  For any commutative square:
\begin{displaymath}
\xymatrix{
	X \ar[d]^{f}\ar[r]^g & X'\ar[d]^{f'} \\
	Y \ar[r]^i & Y'}
\end{displaymath}  
defined over a fixed scheme $S$, the natural diagram:
\begin{displaymath}
\xymatrix{
Lg^*\Cot_{X'/Y'} \ar[d]^{dg}\ar[rr]^-{g^*K_{X'/Y'/S}} & & Lg^*Lf'^*\Cot_{Y'/S}[1] \isom Lf^*Li^*\Cot_{Y'/S}[1] \ar[d]^{Lf^*di} \\
	\Cot_{X/Y} \ar[rr]^-{K_{X/Y/S}} & &  Lf^*\Cot_{Y/S}[1] 
	 }
\end{displaymath}
is commutative, where the right most vertical arrow is the canonical morphism $Lf^*di$ induced by the equality $if' = fj$.

\subsubsection{Lifting schemes - functoriality of obstruction classes}

We now recall the obstruction theory expressed in terms of cotangent complex.  We consider $i : Y \to \wt{Y}$, a square-zero $S$-extension of $Y$ by an $\cc{O}_Y$-module $\cc{J}$. Apart from functoriality all the results are given in \cite{cotangent}.  We begin with a classification of $S$-extensions.

\begin{thm}
There exists a bijection between the group $\mathrm{Exal}_S(Y,\cJ)$ of isomorphism classes of $S$-extension of $Y$ by an $\cO_Y$-module $\cJ$ and the group $\Ext^1(\Cot_{Y/S},\cJ)$. 
\end{thm}
\begin{proof}
See \cite[Part I, Chapitre III, Th\'{e}or\`{e}me 1.2.3]{cotangent}.
\end{proof}

\begin{defin}[Deformation tuple]
A \emph{deformation tuple over $S$} consists of:
\begin{enumerate}[a)]
\item a diagram of schemes
\begin{displaymath}
\xymatrix{
	X \ar[d]^{f} & &  \\
	Y \ar[rr]^i\ar[d] & & \wt{Y}\ar[dll] \\
	S,}
\end{displaymath}  
where $\wt{Y}$ is a square-zero $S$-extension of $Y$ by a module $\cc{J}$.
\item an $\cc{O}_X$-module $\cc{I}$ and a module homomorphism $w : f^*\cc{J} \to \cc{I}$.
\end{enumerate}
\end{defin}

\begin{thm}\label{thm:obstruction_schemes}
For any deformation tuple as above, there exists an obstruction $\sigma_X \in \Ext^2(\Cot_{X/Y},\cc{I})$ whose vanishing is sufficient and necessary for the existence of an extension $\wt{X}$ as in the diagram:
\begin{displaymath}
\xymatrix{
	X \ar[d]^{f}\ar@{-->}[rr] & & \wt{X}\ar@{-->}[d] \\
	Y \ar[rr]^i\ar[d] & & \wt{Y}\ar[dll] \\
	S,}
\end{displaymath}
inducing $w : f^*\cJ \to \cI$ on the kernels of thickenings.  The obstruction $\sigma_X$ is given by the composition:
\begin{displaymath}
\xymatrix{
	\Cot_{X/Y} \ar[rr]^-{K_{X/Y/S}} & & Lf^*\Cot_{Y/S}[1] \ar[rr]^{Lf^*\delta[1]} & & Lf^*\cc{J}[2] \ar[r]^-{H^0} & f^*\cc{J}[2]\ar[r]^-{w[2]} & \cc{I}[2],}
\end{displaymath}
where $\delta \in \Ext^1(\Cot_{Y/S},\cc{J})$ denotes the cohomology class associated to the extension $\wt{Y}$.  If the obstruction vanishes, the liftings constitute a torsor under $\Ext^1(\Cot_{X/S},\cI)$.
In particular, the obstruction class for the existence of a flat lifting $\wt{X} \to \wt{Y}$ is an element of $\Ext^2(\Cot_{X/Y},f^*\cJ)$.
\end{thm}
\begin{proof}
For the proof see \cite[Part I, Chapitre III, Th\'{e}or\`{e}me 2.1.7]{cotangent}.  
\end{proof}

Liftability obstructions satisfy the following functoriality property (for the sake of convenience we assume the ideals of thickenings are flat).

\begin{defin}[Morphism of deformation tuples]
A \emph{morphism of deformation tuples} over $S$ is a piece of data denoted diagramatically by:
\[
\begin{bmatrix}
    \xymatrix{
	X \ar[d]^{f} & & \\
	Y \ar[rr]^i\ar[d] & & \wt{Y}\ar[dll] \\
	S,} \\ \\ 
	(\cI,\cJ,w : f^*\cc{J} \to \cc{I})
\end{bmatrix} \xrightarrow[\text{deformation tuples}]{\text{a morphism of}}
\begin{bmatrix}
    \xymatrix{
	X' \ar[d]^{f'}  & & \\
	Y' \ar[rr]^{i'}\ar[d] & & \wt{Y}'\ar[dll] \\
	S} \\ \\  
	(\cI',\cJ',w' : f'^*\cc{J}' \to \cc{I}'),
\end{bmatrix}
\]
and consisting of:
\begin{enumerate}[a)]
\item a diagram of schemes over $S$:
\begin{displaymath}
    \xymatrix{
         & X \ar[dl]^-{g}\ar[d]^-{f} &  \\
        X'\ar[d]^-{f'} & Y \ar[rr]^i\ar[ld]^-{h} & & \wt{Y}\ar[dl]^-{\wt{h}} \\
        Y' \ar[rr]^{i'} &  &\wt{Y}', &
        }  
\end{displaymath} 
where $\wt{h}$ is a morphism inducing an $\cc{O}_Y$-module homomorphism $u : h^*\cc{J}' \to \cc{J}$,
\item a homomorphism $v : g^*\cc{I}' \to \cc{I}$ fitting into a commutative diagram:
\begin{displaymath}
    \xymatrix{
g^*f'^*\cc{J}' \isom f^*h^*\cc{J}' \ar[d]^{f^*u}\ar[rr]^-{g^*w'} & & g^*\cc{I}' \ar[d]^{v} \\
f^*\cc{J}\ar[rr]^-{w} & & \cc{I}. }
\end{displaymath}
\end{enumerate}
\end{defin}

\begin{lemma}[Functoriality of obstructions to lifting schemes]\label{lem:obstruction_schemes_func}
For any morphism of deformation tuples as above, the obstruction classes to existence of liftings fit in the commutative diagrams:
\begin{displaymath}
    \xymatrix{
         Lg^*\Cot_{X'/Y'} \ar[r]^{g^*\sigma_{X'}}\ar[d]^{dg}	& Lg^*\cc{I}' \ar[d]^{v} 	& &\Cot_{X'/Y'} \ar[r]^{\sigma_{X'}}\ar[d]^{dg} & \cc{I}'[2] \ar[d] \\
        \Cot_{X/Y} \ar[r]^-{\sigma_X} 	& \cc{I}, 								& &Rg_*\Cot_{X/Y} \ar[r]^-{Rg_*\sigma_X} & Rg_*\cc{I}[2].}  
\end{displaymath}
\end{lemma}
\begin{proof}
We use scommutativity of the diagram:
{\scriptsize
\begin{displaymath}
\xymatrix{
	& & & & & \\
	Lg^*\Cot_{X'/Y'} \ar[d]^{dg}\ar@/_-2.4pc/[rrrr]+<0ex,2ex>^{Lg^*\sigma_{X'}}\ar[rr]^-{g^*K_{X'/Y'/S}} 	& & Lg^*Lf'^*\Cot_{Y'/S}[1] \isom Lf^*Lh^*\Cot_{Y'/S}[1] \ar[r]\ar[d]^{Lf^*di} & Lg^*Lf'^*\cc{J'}[2]\ar[d]^{f^*u}\ar[r]^-{Lg^*H^0} & Lg^*f'^{*}\cc{J'}[2] \ar[r]^-{Lg^*w'} & Lg^*\cc{I}' \ar[d]^-{v \circ H^0} \\ 
	\Cot_{X/Y} \ar@/_1.6pc/[rrrrr]^{\sigma_X}\ar[rr]^-{K_{X/Y/S}} 	& & Lf^*\Cot_{Y/S}[1] \ar[r] 	& Lf^*\cc{J}[2]\ar[r]^{H^0}	& f^*\cc{J}[2]\ar[r]^-{w} & \cc{I},
	}
\end{displaymath}
}
which follows from: 
\begin{enumerate}[a)]
\item the description of obstructions classes,
\item functoriality of Kodaira-Spencer class (left-most square),
\item commutativity of the lower part of the diagram in the definition of morphism of deformation tuples (middle square),
\item the assumption on the homomorphisms of flat ideals (right-most square).
\end{enumerate}
We consequently derive the following diagram which is constituted by boundary arrows.
\begin{displaymath}
    \xymatrix{
            Lg^*\Cot_{X'/Y'} \ar[r]^{g^*\sigma_{X'}}\ar[d]^{dg}	& Lg^*\cc{I}' \ar[d]^{v} \\
        \Cot_{X/Y} \ar[r]^-{\sigma_X} 	& \cc{I}.}  
\end{displaymath}
After application of $Rg_*$ and the natural adjunction this yields the commutativity of:
\begin{displaymath}
    \xymatrix{
    \Cot_{X'/Y'} \ar[rr]^{\sigma_{X'}}\ar[d]^{\eta_{\Cot_{X'/Y'}}} & & \cc{I}'[2]\ar[d]^{\eta_{\cc{I}'[2]}} \\
            Rg_*Lg^*\Cot_{X'/Y'} \ar[rr]^{Rg_*Lg^*\sigma_{X'}}\ar[d]^{Rg_*dg}	& & Rg_*Lg^*\cc{I}'[2] \ar[d]^{Rg_*v} \\
        Rg_*\Cot_{X/Y} \ar[rr]^-{Rg_*\sigma_X} & & Rg_*\cc{I}[2],}  
\end{displaymath}
where $\eta_\cc{A}$ is the counit of the adjunction for an object $\cc{A}$.  This gives the desired result.
\end{proof}

\subsubsection{Lifting morphisms}

Here, we proceed with some facts concerning liftability of morphisms, in particular, we analyse the functoriality properties of obstruction classes.
Firstly, we recall the following theorem describing obstructions to lifting of morphisms.

\begin{defin}[Deformation of morphism tuple]
A \emph{deformation of morphism tuple} over a square-zero extension $Z \to \wt{Z}$ by a flat $\cO_Z$-module $\cc{K}$ consists of:
\begin{enumerate}[a)]
\item a diagram of schemes:
\begin{displaymath}
    \xymatrix{ 
       	X \ar[rr]^i\ar[d]_{f}\ar@/^-1.25pc/[dd]_{h} & & \wt{X} \ar@{-->}[d]\ar@/^+1.25pc/[dd]^{\wt{h}} \\ 
		X' \ar[rr]^j\ar[d]_{g} & &  \wt{X}'\ar[d]_{\wt{g}}	\\
		Z \ar[rr] & & \wt{Z},
		}
\end{displaymath}
where $i$ (resp. $j$) is an $\wt{Z}$-extension of $X$ (resp. $X'$) by an $\cc{O}_X$-module $\cc{I}$ (resp. $\cc{O}_{X'}$-module $\cc{J}$).
\item a homomorphism $w : f^*\cc{J} \to \cc{I}$ satisfying the equality $w_h = w \comp f^*w_g $ where $w_g : g^*\cK \to \cJ$ and $w_h : h^*\cK \to \cI$ are the natural maps induced by $\wt{g}$ and $\wt{h}$, respectively.
\end{enumerate}
\end{defin}

\begin{thm}
For every deformation of morphism tuple as above, there exists an obstruction $\sigma_f \in \Ext^1(f^*\Cot_{X'/Z},\cc{I})$ whose vanishing is sufficient and necessary for the existence of a lifting $\wt{f} : \wt{X} \to \wt{X}'$ inducing $w$ on the level of ideals of $X$ and $X'$.  The obstruction is given by the unique preimage under the natural mapping $\Ext^1(f^*\Cot_{X'/Z},\cc{I}) \ra \Ext^1(f^*\Cot_{X'/\wt{Z}},\cc{I})$ of the difference $w \comp f^*e_j - e_i \comp df$ of arrows in the diagram:
\begin{displaymath}
\xymatrix{
	f^*\Cot_{X'/\wt{S}} \ar[r]^{f^*e_j}\ar[d]^{df} & f^*\cc{J} \ar[d]^w \\
	\Cot_{X/\wt{S}} \ar[r]^{e_i} & \cc{I},}
\end{displaymath}
where $e_i \in \Ext^1(\Cot_{X/\wt{Z}},\cc{I})$, resp. $e_j \in \Ext^1(\Cot_{X'/\wt{Z}},\cc{J})$ are classes of an $\wt{Z}$-extension $\wt{X}$, resp. $\wt{X}'$.  
\end{thm}
\begin{proof}
For the proof see \cite[Part I, Chapitre III, Proposition 2.2.4]{cotangent}.  
\end{proof}

Morphism liftability obstructions satisfy the following functoriality property. 

\begin{defin}[Morphism of deformation of morphism tuples]
A \emph{morphism of deformation of morphism tuples} over a square-zero extension $S \to \wt{S}$ by an $\cO_S$-module is a piece of data denoted diagramatically by:
\[
\begin{bmatrix}
    \xymatrix{ 
       	X' \ar[rr]^{i'}\ar[d]_{f'}\ar@/^-1.25pc/[dd]_{h'} & & \wt{X}' \ar@{-->}[d]\ar@/^+1.25pc/[dd]^{\wt{h}'} \\ 
		Y' \ar[rr]^{j'}\ar[d]_{g'} & &  \wt{Y}' \ar[d]_{\wt{g}'}	\\
		Z \ar[rr] & & \wt{Z}
		} \\ \\  
	(\cI',\cJ',w' : f'^*\cc{J}' \to \cc{I}')
\end{bmatrix}
\xrightarrow[\text{of morphism tuples}]{\text{a morphism of deformation}}
\begin{bmatrix}
    \xymatrix{ 
       	X \ar[rr]^i\ar[d]_{f}\ar@/^-1.25pc/[dd]_{h} & & \wt{X} \ar@{-->}[d]\ar@/^+1.25pc/[dd]^{\wt{h}} \\ 
		Y \ar[rr]^j\ar[d]_{g} & &  \wt{Y}\ar[d]_{\wt{g}}	\\
		Z \ar[rr] & & \wt{Z}
		} \\ \\ 
	(\cI,\cJ,w : f^*\cc{J} \to \cc{I})
\end{bmatrix}
\]
and given by the diagram of schemes
\begin{displaymath}
    \xymatrix{ 
    		& X'\ar[rr]\ar[ld]_k\ar[dd]^(.30){f'} & & \wt{X}'\ar[ld]_{\wt{k}} \ar@{-->}[dd] 	\\
       	X \ar[rr]\ar[dd]_{f} & & \wt{X} \ar@{-->}[dd] & \\ 
		& Y'\ar[ld]_l\ar[rr]\ar[ldd] & & \wt{Y'} \ar[ld]_{\wt{l}}\ar[ldd] \\
		Y \ar[rr]\ar[d]& & \wt{Y}\ar[d]	\\
		S \ar[rr] & & \wt{S} &  
		}
\end{displaymath}
such that the mapping $q_k : k^*\cI \to \cI'$ and $q_l : l^*\cJ \to \cJ'$ fit into a commutative diagram:
\begin{displaymath}
    \xymatrix{
k^*f^*\cJ \isom f'^*l^*\cJ  \ar[d]^{k^*w}\ar[rr]^-{f'^*q_l} & & f'^*\cJ' \ar[d]^{w'} \\
k^*\cI \ar[rr]^-{q_k} & & \cI'. }
\end{displaymath}
\end{defin}

We need the result in limited generality and therefore for the notational convenience we only state the case of flat extensions over a given $\wt{Z}$.

\begin{lemma}[Functoriality of obstructions to lifting morphisms]\label{lem:obstruction_morphisms_func}
For any morphism of deformation of morphism tuples as above where the associated ideals are flat the obstructions $\sigma_f$ and $\sigma_{f'}$ satisfy the relation:
\[
q_k \comp k^*\sigma_f = \sigma_{f'} \comp f'^*dl \comp u^{-1}
,\] where $u : f'^*l^*\Cot_{X'/\wt{Z}} \to k^*f^*\Cot_{X'/\wt{Z}}$ is the canonical isomorphism.
\end{lemma}
\begin{proof}
We observe that the obstructions fit into the following diagram:
\begin{displaymath}
    \xymatrix{
    	f'^*l^*\Cot_{Y/\wt{Z}}\ar@/_-1.4pc/[rrrr]^{f'^*l^*e_j}\ar[r]^{u}_{\isom}\ar[d]^{f'^*dl} & k^*f^*\Cot_{Y/\wt{Z}}\ar @{} [drr] |{=} \ar[d]^(.40){k^*df}\ar[rr]^{k^*f^*e_j} & & k^*f^*\cJ \ar[d]^{k^*w} \ar[r]_{\isom} & f'^*l^*\cJ\ar[d]^{f'^*q_l} \\
		f'^*\Cot_{Y'/\wt{Z}}\ar[dr]^{df'}\ar@/^1.4pc/[rrrr]   & k^*\Cot_{X/\wt{Z}}\ar[d] \ar[rr]^{k^*e_i} 			 & & k^*\cc{I}\ar[d]^{q_k} & f'^*\cc{J}'\ar[ld]^{w'} \\
		& \Cot_{X'/\wt{Z}}\ar[rr]^{e_{i'}} & & \cI', & 
		}
\end{displaymath}

where the middle square is the pullback of $\sigma_f$ along $k$ and the bottom boundary mappings give rise to the obstruction $\sigma_{f'}$. The functoriality is a consequence of
\begin{enumerate}[a)]
\item the description of obstructions classes,
\item commutativity of the left square coming from transitivity of differentials for $fk = lf'$,
\item commutativity of the right square coming from the definition of morphism of deformations of morphism tuples,
\item commutativity of the lower and boundary upper square which in fact constitute the obstructions for existence of $\wt{k}$ and $\wt{l}$, respectively.
\item flatness of the ideals, which allows us to omit derived pull-backs.
\end{enumerate}
\end{proof}


\subsection{Deformations of affine schemes}\label{appendix:affine_deformations}

In this section we present a few results concerning deformation theory of affine schemes.  In particular, we prove that one can check the existence of a $W_2(k)$-lifting \'{e}tale locally.

\begin{lemma}\label{lem:affine_deformation_etale_local}
Let $X \isom \Spec(A)$ be an affine scheme essentially of finite type over $k$.  Assume there exists an \'{e}tale surjective covering $p : U \to X$ such that $U$ is $W_2(k)$-liftable (respectively $F$-liftable).  Then, $X$ is $W_2(k)$-liftable (respectively $F$-liftable).
\end{lemma}
\begin{proof}
We treat the case of $W_2(k)$-liftability as the proof in the case of $F$-liftability is analogous.  By taking an affine Zariski covering of $U$ we may assume that $U \isom \Spec(B)$.  Let $\sigma_A \in \Ext^2(\Cot_{A/k},A)$ and $\sigma_B \in \Ext^2(\Cot_{B/k},B)$ be the obstruction classes to the existence of a $W_2(k)$-lifting of $\Spec(A)$ and $\Spec(B)$, respectively.  By \cref{lem:obstruction_schemes_func} and the properties of the morphism $p$ we see that $\Ext^2(\Cot_{B/k},B) \isom \Ext^2(\Cot_{A/k},A) \otimes_A B$ and $\sigma_B = \sigma_A \otimes 1$ under this isomorphism.  Therefore, the claim of the lemma follows from the fact that $p$ is \'{e}tale and surjective and hence faithfully flat.
\end{proof}

Along the same lines we can prove the following.

\begin{lemma}\label{lem:affine_deformation_Zariski_local}
Let $X$ be an affine scheme locally of finite type over $k$.  Assume that for any closed point $x \in X$ the scheme $\Spec(\cO_{X,x})$ is $W_2(k)$-liftable (respectively $F$-liftable).  Then, $X$ is $W_2(k)$-liftable (respectively $F$-liftable).
\end{lemma}


\subsection{Deformation functors and their morphisms}\label{appendix:deformation_functors}

Here, we recall the formalism of functors of Artin rings.  The classical reference for the topic is \cite{schlessinger_functors}.  The results concerning obstruction theories are well-explained in \cite{fantechi_manetti}.

Suppose $k$ be a field, $\Lambda$ a complete local ring and let $\Art(\Lambda,k)$ denote the category of local Artinian $\Lambda$-algebras with residue field $k$. We consider covariant functors $F : \Art(\Lambda,k) \to \Sets$, which naturally represent infinitesimal deformations of algebraic and geometric objects. The case we are mostly interested in is $\Lambda = W_2(k)$ where $k$ is a perfect field of positive characteristic $p$.

\begin{defin}[Functor of Artin rings]
A \emph{functor of Artin rings} is a covariant functor $F : \Art(\Lambda,k) \to \Sets$ such that $F(k)$ consists of a single element.  All such functors form a category with natural transformations as morphisms.
\end{defin}

Any surjective morphism in $\Art(\Lambda,k)$ can be decomposed into a sequence of \emph{small extensions}, i.e., ring surjections $(B,\ideal{m}_B) \ra (A,\ideal{m}_A)$ with kernel $I$ satisfying $\ideal{m}_B I = 0$.  Small extensions naturally form a category $\sex(\Lambda,k)$ whose objects are small extensions and morphisms are diagrams:
\begin{displaymath}
    \xymatrix{
        0  \ar[r] & I \ar[r]\ar[d]^\phi & B \ar[r]\ar[d]^{f_2} & A \ar[r]\ar[d]^{f_1} & 0 \\
        0  \ar[r] & I' \ar[r] & B' \ar[r] & A' \ar[r] & 0, }  
\end{displaymath}
where $f_1, f_2$ are rings homomorphisms and $\phi$ is an induced homomorphism of $k$-vector spaces.

\begin{defin}[Smooth morphism of functors]
A natural transformation $\xi : F \to G$ of functors of Artin rings is \emph{smooth} if for every surjective morphism $B \to A$ in $ \Art(\Lambda,k)$ the natural mapping:
\[
F(B) \to F(A) \times_{G(A)} G(B)
\]
is surjective.
\end{defin}

From the special case $A = k$, we see that a smooth morphism is surjective.

\begin{defin}[Tangent space and morphism]
Suppose $\cF : \Art(\Lambda,k) \to \Sets$ is a functor of Artin rings.  The set $\Tan_\cF \mydef \cF(k[\varepsilon])$ is called the \emph{tangent space} of a functor $\cF$.  For any morphism of deformation functors $\psi : \cF \to \cG$ the mapping $\Tan_\psi \mydef \psi_{k[\eps]} : \Tan_\cF = \cF(k[\eps]) \to \cG(k[\eps]) = \Tan_\cG$ is called the tangent map of $\psi$.
\end{defin}

Under certain condition (see \cite{fantechi_manetti}), satisfied in any of our applications, the tangent space admits a natural structure of a $k$-vector space such that the tangent morphism becomes $k$-linear.

\begin{defin}[Obstruction theory]
Suppose $\cF : \Art(\Lambda,k) \to \Sets$ is a functor of Artin rings.  An \emph{obstruction theory} for $\cF$ is a pair $(V,\{\nu_e\}_{e \in \sex(\Lambda)})$ of a vector space $V$ and a family of mappings $\nu_e : \cF(A) \to V \otimes_k I$ parametrised by infinitesimal extensions $e: 0 \ra I \ra B \ra A \ra 0$ and satisfying the following properties:
\begin{enumerate}[i)]
\item (functoriality) for any morphism of small extensions
\begin{displaymath}
    \xymatrix{
        e: 0  \ar[r] & I \ar[r]\ar[d]^\phi & B \ar[r]\ar[d]^{f_2} & A \ar[r]\ar[d]^{f_1} & 0 \\
        e': 0  \ar[r] & I' \ar[r] & B' \ar[r] & A' \ar[r] & 0, }  
\end{displaymath}
we have $(\id_V \otimes \phi) \comp \nu_e = \nu_{e'} \comp \cF(f_1)$.  
\item (completeness) for any infinitesimal extension $e: 0 \ra I \ra B \ra A \ra 0$ and an element $a \in \cF(A)$ the condition $\nu_e(a) = 0$ is equivalent to existence of $b \in \cF(B)$ lifting $a$.
\end{enumerate}  
\end{defin}

The choice of obstruction theory is not unique, e.g., any proper inclusion $i : V \to V'$ gives rise to a different obstruction theory $(V',i \comp \nu_e)$.

We now present a criterion for smoothness of morphism of functors in terms of their tangent and obstruction spaces.  Let $(\cF,(\Ob_\cF,\nu^\cF_e))$ and $(\cG,(\Ob_\cG,\nu^\cG_e))$ be deformation functors together with associated obstruction theories and $\psi : \cF \to \cG$ be a morphism.  We say that a linear mapping $c : \Ob_\cF \to \Ob_\cG$ is an obstruction map of $\psi$ if for every infinitesimal extension $e: 0 \ra I \ra B \ra A \ra 0$ we have $\nu^\cG_e \comp \psi_A = (c \otimes_k \id_I) \comp \nu^\cF_e$.  We emphasize that the notion of obstruction map of $\psi$ does not depend solely on $\psi$ but also on the choice of obstruction theories for $\cF$ and $\cG$.

We have the following criterion.

\begin{lemma}[Smoothness criterion for morphism of functors]\label{lem:smoothness_criterion}
Suppose $(\cF,(\Ob_\cF,\nu^\cF_e))$ and $(\cG,(\Ob_\cG,\nu^\cG_e))$ are deformation functors together with associated obstruction theories.  Let $\psi : \cF \to \cG$ be a morphism of functors admitting an obstruction map $\Ob_\psi : \Ob_\cF \to \Ob_\cG$. Then, $\psi$ is smooth if the following conditions hold:
\begin{enumerate}[i)]
\item $\Tan_\psi : \Tan_\cF \to \Tan_\cG$ is surjective,
\item $\Ob_\psi : \Ob_\cF \to \Ob_\cG$ is injective.
\end{enumerate}
\end{lemma}
\begin{proof}
See \cite[Lemma 6.1]{fantechi_manetti}.
\end{proof}

We now give two examples relevant in what follows.

\begin{example}[Abstract deformations of a scheme]
Suppose $X$ is a scheme over a field $k$.  The functor of \emph{abstract deformations} $\Def_X : \Art(\Lambda,k) \to \Sets$ is defined by the assignment:
\[
\Art(\Lambda,k) \ni A \mapsto \left\{ \text{ isomorphism classes of flat deformations of $X$ over $\Spec(A)$ }\right\},
\]
where a deformation consists of a flat scheme $\cc{X}$ over $\Spec(A)$ together with an isomorphism $\phi_{\cc{X}} : \cc{X} \times_{\Spec(A)} \Spec(k) \to X$.  Morphism of deformations $(\cc{X},\phi_{\cc{X}}) \to (\cc{X}',\phi_{\cc{X}'})$ is a morphism of schemes $\cc{X} \to \cc{X}'$ such that its restriction to the special fibre commutes with isomorphisms $\phi$.

By the results of \cref{thm:obstruction_schemes} the tangent space of the abstract deformation functors $\Def_X$ is equal to $\Ext^1(\Cot_{X/k},\cO_X)$.  Moreover, it admits an obstruction theory with the obstruction space equal to $\Ext^2(\Cot_{X/k},\cO_X)$.
\end{example}

\begin{example}[Embedded deformations of a projective scheme]
Suppose $X \subset \PP^n_k$ is a closed subscheme.  The functor of \emph{embedded deformations} $\Hilb_{X,\PP^n} : \Art(\Lambda,k) \to \Sets$ is defined by the assignment:
\[
\Art(\Lambda,k) \ni A \mapsto \left\{ 
\begin{gathered}
\text{ isomorphism classes of closed subschemes of $\PP^n_A$ } \\
\text{ flat over $\Spec(A)$ and restricting to $X$ over $\Spec(k)$ }
\end{gathered}
\right\}.
\]
By \cite[Chapter I]{hartshorne_deformation} the tangent space of $\Hilb_{X,\PP^n}$ is naturally isomorphic to $H^0(X,\cN_{X/\PP^n})$.  Furthermore, it admits an obstruction theory with the obstruction space equal to $H^1(X,\cN_{X/\PP^n})$.
\end{example}

We now present an exemplary application of \cref{lem:smoothness_criterion}.

\begin{example}
For any closed embedding $X \subset \PP^n$ of a smooth scheme $X/k$, we obtain a natural morphism of functors $\psi : \Hilb_{X,\PP^n} \to \Def_X$ admitting an obstruction morphism defined by taking the underlying abstract deformation of an embedded deformation.  Its tangent and obstruction mappings:
\begin{align*}
\Tan_\psi : \Tan_{\Hilb_{X,\PP^n}} = H^0(X,\cN_{X/\PP^n}) \to H^1(X,\cT_X) \isom \Ext^1(\Cot_{X/k},\cO_X) = \Tan_{\Def_X} \\
\Ob_\psi : \Ob_{\Hilb_{X,\PP^n}} = H^1(X,\cN_{X/\PP^n}) \to H^2(X,\cT_X) \isom \Ext^2(\Cot_{X/k},\cO_X) = \Ob_{\Def_X} 
\end{align*}
are given by the canonical morphisms coming from the long exact sequence of cohomology associated to:
\[
0 \ra \cT_X \ra \cT_{\PP^n |X} \ra \cN_{X/\PP^n} \ra 0.
\]
As a simple corollary of \cref{lem:smoothness_criterion} we have the following lemma applicable for example to projective Calabi-Yau varieties of $\dim X \geqslant 3$.

\begin{lemma}\label{lem:comparing_hilb_def}
Suppose $X$ is a projective scheme satisfying $H^2(X,\cO_X) = 0$.  For sufficiently positive embedding of $X \subset \PP^n$ the morphism $\psi : \Hilb_{X,\PP^n} \to \Def_X$ is smooth.  In particular, every abstract deformation arises as an embedded one.
\end{lemma}
\begin{proof}
By \cref{lem:smoothness_criterion} we need to show that $H^0(X,\cN_{X/\PP^n}) \to H^1(X,\cT_X)$ is surjective and $H^1(X,\cN_{X/\PP^n}) \to H^2(X,\cT_X)$ is injective for sufficiently ample embedding $X \to \PP^n$.  By the long exact sequence of cohomology it suffices to show that $H^1(X,\cT_{\PP^n|X}) = 0$.  By restricting the Euler sequence to $X$ we obtain:
\[
0 \ra \cO_{X} \ra \cO_{X}(1)^{\oplus n+1} \ra \cT_{\PP^n|X} \ra 0.
\]
By another long exact sequence of cohomology we see that: 
\[
\cdots \ra H^1(X,\cO_{X}(1)^{\oplus n+1}) \ra H^1(X, \cT_{\PP^n|X}) \ra H^2(X,\cO_X) \ra \cdots,
\]
which gives the claim by the Serre's vanishing, i.e., for any coherent sheaf $\cF$ on $X$ we have $H^1(X,\cF(1)) = 0$ if the embedding $X \to \PP^n$ is sufficiently ample.
\end{proof}
\end{example}

\begin{remark}
In particular, for sufficiently positive Veronese embedding of a smooth scheme $X$ satisfying $H^2(X,\cO_X) = 0$, the mapping of functors $\psi_d : \Hilb_{X,\PP^{N_d}} \to \Def_X$ is smooth.
\end{remark}

We finish with a lemma comparing abstract deformation functors of a scheme $X$ and its open subset $U$.  We give a proof based on the cotangent complex.  For the classical proof the reader is encouraged to see \cite[Proposition 9.2]{artin_lectures}. 

\begin{lemma}\label{lem:deformation_opens}
Let $j : U \to X$ be the inclusion of an open subset.  Assume, that $X$ satisfies property $S_3$ at any point $\ideal{p}$ of the complement $Z = X \setminus U$.  Then, the natural morphism of deformation functors $\chi_j : \Def_X \to \Def_U$ coming from the restriction is smooth.
\end{lemma}
\begin{proof}
By the long exact sequence of $\Ext(\Cot_{X/k},-)$ groups associated to a distinguished triangle:
\[
\cO_X \ra Rj_*\cO_U \ra K_j \ra \cO_X[1]
\]
we obtain a sequence:
\begin{align*}
\ldots & \ra \Ext^0(\Cot_{X/k},K_j) \ra \Ext^1(\Cot_{X/k},\cO_X) \xrightarrow{t_j} \Ext^1(\Cot_{X/k},Rj_*\cO_U) \isom \Ext^1(\Cot_{U/k},\cO_U) \ra \\
& \ra \Ext^1(\Cot_{X/k},K_j) \ra \Ext^2(\Cot_{X/k},\cO_X) \xrightarrow{o_j} \Ext^2(\Cot_{X/k},Rj_*\cO_U) \isom \Ext^2(\Cot_{U/k},\cO_U) \ra \ldots,
\end{align*}
where the mappings $t_j$ and $o_j$ can be identified with $\Tan_{\chi_j}$ and $\Ob_{\chi_j}$.  Therefore, by \cref{lem:smoothness_criterion}, it suffices to prove that $\Ext^1(\Cot_{X/k},K_j) = 0$.
This is a direct consequence of the spectral sequence 
\[
\Ext^p(\Cot_{X/k},\cH^q(K_j)) \Rightarrow \Ext^{p+q}(\Cot_{X/k},K_j)
\]
and the following lemma:
\begin{lemma}[Local cohomology in terms of $K_j$]
For any $i \in \ZZ$ the sheaf $\cH^i(K_j)$ is isomorphic to the local cohomology sheaves $\wt{U \mapsto H^{i+1}_{Z}(U,\cO_U)}$.  Consequently, if $X$ satisfies property $S_{n+1}$ at any point of $Z = X \setminus U$ of codimension at least $n+1$ then $\cH^i(K_j) = 0$ for $i \leq n$. 
\end{lemma}
\begin{proof}
For the proof see \cite[\href{http://stacks.math.columbia.edu/tag/0A39}{Tag 0A39}]{stacks-project} and \cite[Proposition 3.3]{hassett_kovacs}.
\end{proof}
This finishes the proof of \cref{lem:deformation_opens}.
\end{proof}


\section*{Acknowledgement}

I would like to thank my advisor Professor Adrian Langer for the choice of the research topic and numerous guiding suggestions.  Moreover, I am grateful to Agnieszka Bodzenta, Joachim Jelisiejew, \L{}ukasz Sienkiewicz and especially Piotr Achinger for their support and many useful discussions.  I was supported by Polish National Science Centre (NCN) contract number 2014/13/N/ST1/02673.




\begin{thebibliography}{{Sta}15}

\bibitem[AC11]{rulling_chatzistamatiou}
K.~R\"{u}lling A.~Chatzistamatiou.
\newblock Higher direct images of the structure sheaf in positive
  characteristic.
\newblock {\em Algebra Number Theory}, 5:693--775, 2011.

\bibitem[AM69]{atiyah_macdonald}
M.~F. Atiyah and I.~G. MacDonald.
\newblock {\em Introduction to commutative algebra}.
\newblock Addison-Wesley-Longman, 1969.

\bibitem[Art76]{artin_lectures}
M.~Artin.
\newblock Lectures on deformations of singularities, 1976.

\bibitem[Art77]{artin_canonical_charp}
M.~Artin.
\newblock {C}overings of the {R}ational {D}ouble {P}oints in {C}haracteristic
  p.
\newblock {\em Complex Analysis and Algebraic Geometry}, 1977.

\bibitem[Bha12]{bhatt_derived_derham}
Bhargav Bhatt.
\newblock p-adic derived de {R}ham cohomology, 2012, arXiv:1204.6560.

\bibitem[Bha14]{bhatt_torsion}
Bhargav Bhatt.
\newblock Torsion in the crystalline cohomology of singular varieties.
\newblock {\em Documenta Mathematica}, 19(22), 2014.

\bibitem[BK05]{brion_kumar}
M.~Brion and S.~Kumar.
\newblock {\em Frobenius {S}plitting {M}ethods in {G}eometry and
  {R}epresentation {T}heory}.
\newblock Progress in Mathematics. Birkh{\"a}user Boston, 2005.

\bibitem[Blo77]{bloch_ktheory}
Spencer Bloch.
\newblock Algebraic {K}-theory and crystalline cohomology.
\newblock {\em Publications Math\'{e}matiques de l'Institut des Hautes
  \'{E}tudes Scientifiques}, 47(1):188--268, 1977.

\bibitem[BO78]{OgusBerthelot}
P.~Berthelot and A.~Ogus.
\newblock {\em Notes on Crystalline Cohomology}.
\newblock Mathematical Notes - Princeton University Press. Princeton University
  Press, 1978.

\bibitem[DI87]{deligne-illusie}
Pierre Deligne and Luc Illusie.
\newblock Relevements modulo p2 et decomposition du complexe de de rham.
\newblock {\em Inventiones mathematicae}, 89:247--270, 1987.

\bibitem[FM81]{fulton_macpherson}
W.~Fulton and R.~MacPherson.
\newblock {\em Categorical Framework for the Study of Singular Spaces}.
\newblock Number no. 243 in American Mathematical Society: Memoirs of the
  American Mathematical Society. American Mathematical Society, 1981.

\bibitem[FM98]{fantechi_manetti}
Barbara Fantechi and Marco Manetti.
\newblock Obstruction calculus for functors of artin rings, i.
\newblock {\em Journal of Algebra}, 202(2):541 -- 576, 1998.

\bibitem[Gla96]{donna_glassbrenner}
Donna Glassbrenner.
\newblock Strong f-regularity in images of regular rings.
\newblock {\em Proceedings of the American Mathematical Society}, 124(2), 1996.

\bibitem[Har09]{hartshorne_deformation}
R.~Hartshorne.
\newblock {\em Deformation Theory}.
\newblock Graduate Texts in Mathematics. Springer New York, 2009.

\bibitem[HK04]{hassett_kovacs}
Brendan Hassett and S{\'a}ndor~J Kov{\'a}cs.
\newblock Reflexive pull-backs and base extension.
\newblock {\em Journal of Algebraic Geometry}, 13(2):233--248, 2004.

\bibitem[Ill72]{cotangent}
Luc Illusie.
\newblock {\em Complexe cotangent et d\'eformations {I} and {II}}.
\newblock Lecture Notes in Mathematics, Vol. 239 and 283. Springer-Verlag,
  Berlin, 1971/1972.

\bibitem[Ill96]{illusie_frobenius}
Luc Illusie.
\newblock {\em Frobenius et d\'{e}g\'{e}n\'{e}rescence de Hodge. In:
  Introduction la th\'{e}orie de Hodge}, volume~3 of {\em Panoramas et
  Synth\`{e}ses}, pages 113--168.
\newblock Soc. Math. France, Paris, 1996.

\bibitem[Jos07a]{joshi}
Kirti Joshi.
\newblock Exotic torsion, {Frobenius} splitting and the slope spectral
  sequence.
\newblock {\em Canadian Mathematical Bulletin}, 50(4):567--587, December 2007.

\bibitem[Jos07b]{joshi_exotic}
Kirti Joshi.
\newblock Exotic torsion, {Frobenius} splitting and the slope spectral
  sequence.
\newblock {\em {C}anadian {M}athematical {B}ulletin}, 50(4):567--587, December
  2007.

\bibitem[Kra05]{ck_det}
Christian Krattenthaler.
\newblock Advanced {Determinant} {Calculus}: {A} {Complement}.
\newblock \href{http://arxiv.org/abs/math/0503507}, 2005.

\bibitem[Lan08]{Langer_2008}
Adrian Langer.
\newblock D-affinity and {F}robenius morphism on quadrics.
\newblock {\em International Mathematics Research Notices}, 145:1--26, 2008.

\bibitem[Lan15]{langer}
Adrian Langer.
\newblock Bogomolov's inequality for {H}iggs sheaves in positive
  characteristic.
\newblock {\em Inventiones mathematicae}, 199(3):889--920, 2015.

\bibitem[LS14]{liedtke_satriano}
Christian Liedtke and Matthew Satriano.
\newblock On the birational nature of lifting.
\newblock {\em Advances in Mathematics}, 254:118--137, 2014.

\bibitem[MS87]{mehta_srinivas}
V.~B. Mehta and V.~Srinivas.
\newblock Varieties in positive characteristic with trivial tangent bundle.
\newblock {\em Compositio Mathematica}, 64(2):191--212, 1987.

\bibitem[Muk13]{mukai}
Shigeru Mukai.
\newblock Counterexamples to {K}odaira's vanishing and {Y}au's inequality in
  positive characteristics.
\newblock {\em Kyoto J. Math.}, 53(2):515--532, 2013.

\bibitem[OV07]{ogus-vologodsky}
Arthur Ogus and Vadim Vologodsky.
\newblock {\em Nonabelian {H}odge theory in characteristic p}.
\newblock Publications math{\'e}matiques. Institut des Hautes Etudes
  Scientifiques, 2007.

\bibitem[Ray78]{raynaud}
Michel Raynaud.
\newblock {\em Contre-example au "vanishing de Kodaira" sur une surface lisse
  en caract\'{e}ristique $p > 0$}, pages 273--278.
\newblock C. P. Ramanujam - A Tribute. Springer Verlag, 1978.

\bibitem[Sch68]{schlessinger_functors}
Michael Schlessinger.
\newblock Functors of artin rings.
\newblock {\em TRANS. A.M.S}, 130(2):208--222, 1968.

\bibitem[Sch71]{schlessinger_rigidity}
Michael Schlessinger.
\newblock Rigidity of quotient singularities.
\newblock {\em Inventiones mathematicae}, 14(1):17--26, 1971.

\bibitem[Sin98]{anuragsingh}
Anurag Singh.
\newblock {P}h{D} {T}hesis: {F}-regularity, {F}-rationality and {F}-purity,
  1998.

\bibitem[{Sta}15]{stacks-project}
The {Stacks Project Authors}.
\newblock {S}tacks {P}roject.
\newblock \href{http://stacks.math.columbia.edu}, 2015.

\bibitem[TW15]{takagi_watanabe}
S.~Takagi and K-I. Watanabe.
\newblock $f$-singularities: applications of characteristic $p$ methods to
  singularity theory, 2015, arXiv:1409.3473.

\end{thebibliography}
\end{document}